\documentclass[12pt,a4paper]{amsproc}

\usepackage[T1]{fontenc}
\usepackage[utf8]{inputenc}
\usepackage{lmodern}
\usepackage{amsmath,amssymb,amsthm,mathtools}
\usepackage{enumitem, todonotes}
\usepackage[
  colorlinks=true,
  linkcolor=blue,
  citecolor=blue,
  urlcolor=blue,
  pagebackref=true
]{hyperref}

\usepackage[
  a4paper,
  left=29mm,
  right=29mm,
  top=29mm,
  bottom=31mm
]{geometry}

\title[The Howe--Moore property]
{On the Howe--Moore property for\\ automorphism groups of buildings}

\author[P.-E. Caprace]{Pierre-Emmanuel Caprace}
\address{P.-E.C., UCLouvain, Chemin du Cyclotron 2, 1348 Louvain-la-Neuve, Belgium}
\email{pierre-emmanuel.caprace@uclouvain.be}

\author[A. Thom]{Andreas Thom}
\address{A.T., TU Dresden, 01062 Dresden, Germany}
\email{andreas.thom@tu-dresden.de}
\date{July 14, 2026}

\thanks{PEC is supported in part by the FWO and the F.R.S.-FNRS under the EOS programme (project ID 40007542).}

\newtheorem{theorem}{Theorem}[section]
\newtheorem{thmintro}{Theorem}

\newtheorem{corintro}[thmintro]{Corollary}
\newtheorem{proposition}[theorem]{Proposition}
\newtheorem{lemma}[theorem]{Lemma}
\newtheorem{corollary}[theorem]{Corollary}

\theoremstyle{definition}

\newtheorem{remark}[theorem]{Remark}

\DeclareMathOperator{\Aut}{Aut}
\DeclareMathOperator{\Stab}{Stab}

\DeclareMathOperator{\supp}{supp}
\DeclareMathOperator{\HS}{HS}
\DeclareMathOperator{\Tr}{Tr}
\DeclareMathOperator{\Ad}{Ad}
\DeclareMathOperator{\ch}{Ch}
\DeclareMathOperator{\IRS}{IRS}

\newcommand{\IRSerg}{\IRS_{\mathrm{erg}}}
\newcommand{\cH}{\mathcal H}
\newcommand{\cK}{\mathcal K}
\newcommand{\cU}{\mathcal U}

\newcommand{\one}{\mathbf 1}
\newcommand{\ol}[1]{\overline{#1}}
\newcommand{\norm}[1]{\left\lVert #1\right\rVert}
\newcommand{\abs}[1]{\left\lvert #1\right\rvert}
\newcommand{\ip}[2]{\left\langle #1,#2\right\rangle}
\newcommand{\weakcontain}{\prec}
\newcommand{\qmin}{q_{\mathrm{min}}}

\begin{document}

\begin{abstract}
Let $G$ be a closed type-preserving subgroup of the automorphism group of a thick locally finite building $X$ of finite rank, and assume that $G$ acts Weyl-transitively. We prove that every   unitary representation of $G$ is mixing, unless its restriction to a parabolic subgroup of minimal non-spherical type is amenable in the sense of Bekka. It follows that every unitary representation of $G$ that is weakly contained in the regular representation, is mixing. In case $X$ is of minimal non-spherical type and its thickness satisfies some modest lower bound, we deduce that $G$  has the Howe--Moore property provided its only compact quotient is trivial. We also  obtain results on  rigidity of invariant random subgroups for Kac--Moody lattices of compact hyperbolic type, yielding examples of infinite finitely presented Kazhdan groups with exactly two  ergodic invariant random subgroups.
\end{abstract}

\maketitle

\section{Introduction}
\label{sec:introduction}

A locally compact group $G$ has the \textbf{Howe--Moore property} if, for every   unitary representation $\pi \colon G\to \cU(\cH)$ without non-zero $G$-fixed vectors, the matrix coefficient $g\mapsto \ip{\pi(g)\xi}{\eta}$ vanishes at infinity for all $\xi,\eta\in\cH$. This property was established by Howe--Moore \cite{HoweMoore1979} for connected non-compact simple real Lie groups with finite center, and for simple algebraic groups over non-Archimedean local fields.
Besides  linear groups, the case of totally disconnected groups acting strongly transitively on trees  was treated by Lubotzky--Mozes \cite{LubotzkyMozes1992} and Burger--Mozes \cite[Prop.~4.2]{BuMo_lattices}; Ciobotaru gave a unified proof for the classical simple groups and the automorphism groups of trees, see \cite{Ciobotaru}.
The initial goal of this paper is to establish the Howe--Moore property for locally compact Kazhdan groups beyond the classical cases. The class groups we consider consists of locally compact groups $G$ acting Weyl-transitively on a building $X$ of minimal non-spherical type $(W, S)$. We show that if the thickness of $X$ is at least~$11$, every weakly mixing unitary representation of $G$ is mixing. This implies that complete Kac--Moody groups of minimal non-spherical type over finite fields of order~$\geq 11$ have the Howe--Moore property, see Corollary~\ref{corollary_HM_KM}. The restriction on the type of the building being minimal non-spherical type is necessary: indeed, without that condition, the group $G$ possesses  parabolic subgroups $P_J$ that are non-compact proper open subgroups. This  is an obstruction to the Howe--Moore property, since the quasi-regular representation on $G/P_J$ is not mixing. Our main result, Theorem~\ref{theoremA}, concerns  Weyl-transitive automorphism groups $G$ of thick locally finite buildings $X$ without any restriction on  the type $(W, S)$: we show that  every   unitary representation of $G$ is mixing, unless its restriction to a parabolic subgroup of minimal non-spherical type is amenable in the sense of Bekka.  The Howe--Moore property in minimal non-spherical type will be derived as a consequence.
 
Before presenting our results in more detail, we  recall  standard terminology on buildings and their automorphisms, following \cite{AbramenkoBrown}. Let $X$ be a thick locally finite building of finite rank and type $(W,S)$ with Weyl distance $\delta \colon \ch(X) \times \ch(X) \to W$, and let $G<\Aut(X)$ be a closed type-preserving subgroup. We assume that the $G$-action is  \textbf{Weyl-transitive}, i.e. it is transitive on ordered pairs of chambers at any given Weyl distance. This condition is formally weaker than strong transitivity, see  \cite[Corollary~6.12]{AbramenkoBrown}. It is equivalent to strong transitivity in the special case of closed subgroups of the automorphism group of a locally finite affine building, see \cite[Theorem~5.16]{Kramer}. Fix a base chamber $c_0$ and put $B=\Stab_G(c_0)$. Then $B$ is compact open and the $B$-double cosets are indexed by the Weyl group $W$. For $J\subset S$, let $W_J=\langle J\rangle$ and let $P_J$ be the stabilizer of the  residue of type $J$ containing $c_0$; equivalently, $P_J=\bigsqcup_{w\in W_J} BwB$, where $BwB$ denotes the double coset of elements sending the chamber $c_0$ to chambers at Weyl distance $w$ from $c_0$. In the strongly transitive case this is the usual standard parabolic subgroup coming from the associated $BN$-pair.
A subset $J\subset S$ is called \textbf{spherical} if $W_J$ is finite. It is called \textbf{minimal non-spherical} if $W_J$ is infinite, while $W_I$ is finite for every proper subset $I\subsetneq J$. We say that $(W, S)$ is \textbf{$n$-spherical} if every subset $J \subset S$ of cardinality~$\leq n$ is spherical. We say that $(W, S)$ is \textbf{simply laced} if for all distinct $s, t \in S$, the order of $st$ is $2$ or $3$ (in particular $(W, S)$ is $2$-spherical). 

In this paper, we study unitary representations of $G$. Unitary representations of locally compact groups are assumed throughout to be strongly continuous. We say that a  unitary representation $\pi \colon G\to\cU(\cH)$ is \textbf{mixing} (or $C_0$) if all its matrix coefficients vanish at infinity, i.e. for all $\xi,\eta\in\cH$, the function $g\mapsto \ip{\pi(g)\xi}{\eta}$ belongs to $C_0(G)$.
We say that $\pi$ is \textbf{amenable in the sense of Bekka} if $\one_G\weakcontain \pi\otimes\ol\pi$, see \cite{BekkaAmenable} for details. Equivalently, the conjugation action of $G$ on $\mathcal B(\cH)$ admits an invariant state.

\medskip

The main result  of this note is the following.

\begin{thmintro}
\phantomsection\label{theoremA}
Let $X$ be a thick locally finite building of finite rank and type $(W,S)$, and let $G<\Aut(X)$ be a closed type-preserving subgroup acting Weyl-transitively. Let $\pi \colon G\to\cU(\cH)$ be a   unitary representation. Then at least one of the following assertions holds.
\begin{enumerate}[label=\textup{(\roman*)},leftmargin=*]
\item $\pi$ is mixing.
\item There is a minimal non-spherical subset $J\subset S$ such that $\pi|_{P_J}$ is amenable in the sense of Bekka.
\end{enumerate}
\end{thmintro}

The assertions (i) and (ii) are not mutually exclusive. Indeed, in case $W$ is infinite dihedral, the building $X$ is a biregular tree, and the group  $G$ has the Haagerup property, see \cite[\S1.2.3]{CCJJV}. By \cite[Theorem~2.1.1]{CCJJV}, this means that $G$ has a mixing representation $\pi$ with $\one_G \weakcontain \pi$. In particular  $\pi$ is Bekka amenable, hence it satisfies both assertions of Theorem~\ref{theoremA}. 

The proof of Theorem~\ref{theoremA} shows that $\pi$ is mixing provided that the diagonal matrix coefficient defined by any $B$-fixed vector in $\pi \otimes \ol \pi$ is mixing (see Proposition~\ref{prop:Bfixed-tensor-implies-C0}, valid for all totally disconnected locally compact groups). The latter condition is obtained by establishing an upper bound $\theta < 1$  on the norm of Hecke operators defined by short words $w $ in the Weyl group $W$ of $G$. This specific step is achieved by a computation in the Iwahori--Hecke-algebra. We consider the   basis  $(T_w)_{w \in W}$ of the Hecke algebra defined by setting $T_w = \frac 1 {m(BwB)} \one_{BwB}$. If $w=s_1\cdots s_n \in W$ is reduced and $J=\supp(w)=\{s_1,\ldots,s_n\}$, then the symmetric probability measure corresponding to $T_w^* * T_w$ is supported on $P_J$ and its support generates $P_J$, see    Lemma \ref{lem:hecke-markov-parabolic}. Thus $T_w^* * T_w$ is a Markov operator for the parabolic subgroup generated by the letters in a reduced expression of $w$. The spectral criterion of Bekka--de la Harpe--Valette \cite[Proposition G.4.2]{BHV} (see Proposition \ref{prop:BHV}) implies that $\|\pi(T_w)\| <1$, unless $\pi|_{P_J}$ is amenable in the sense of Bekka. 

Theorem~\ref{theoremA} has particularly strong consequences in the case where  the minimal non-spherical parabolics enjoy Kazhdan's property (T): in that case the two assertions of Theorem~\ref{theoremA} are mutually exclusive, and we get a strict dichotomy. Coxeter groups of  minimal non-spherical type are either affine, or of compact hyperbolic type and rank~$3, 4$ or $5$. In the affine case, property (T) always holds (see \cite{Oppenheim}). In   compact hyperbolic type,  we provide modest lower bounds on the thickness guaranteeing property (T), see Proposition~\ref{prop:Kazhdan}. Combining those with Theorem~\ref{theoremA}, we obtain the following. 

\begin{corintro}
\phantomsection\label{corollaryB}
Let $X$, $G$ and $(W,S)$ be as in Theorem~\ref{theoremA}. Suppose that $(W, S)$ satisfies at least  one of the following conditions:
\begin{enumerate}[label=(\arabic*)]
\item $(W, S)$ is simply laced. 
\item $(W, S)$ is $3$-spherical, and every panel of $X$  contains at least $5$ chambers. 
\item $(W, S)$ is $2$-spherical crystallographic, and every panel of $X$ contains  at least $11$ chambers.
\end{enumerate}
Then every unitary representation $\pi \colon G\to\cU(\cH)$ satisfies the following dichotomy: either $\pi$ is mixing, or $\pi|_{P_J}$ has a 
finite-dimensional subrepresentation for some minimal non-spherical $J\subset S$. 
\end{corintro}

If $(W, S)$ is of  minimal non-spherical type, the only minimal non-spherical parabolic subgroup is $G$ itself. Thus, if property (T) holds  for $G$, Theorem~\ref{theoremA} gives decay of matrix coefficients for all weakly mixing representations. If moreover every finite-dimensional unitary representation of $G$ is trivial,   we get the classical Howe--Moore property.

\begin{corintro}
\phantomsection\label{corollaryBB}
Let $X$, $G$ and $(W,S)$ be as in Corollary~\ref{corollaryB}. Suppose moreover that $(W, S)$ is minimal non-spherical. If   every compact quotient of $G$ is trivial, then  $G$ has the Howe--Moore property.
\end{corintro}

For a group $G$ acting Weyl-transitively on a thick irreducible locally finite building, the absence of non-trivial compact quotients is equivalent to the topological simplicity of $G$, see Proposition~\ref{prop:simple-by-compact}. Concrete examples of locally compact groups satisfying those conditions are provided by complete Kac--Moody groups (see \cite[Chapter~8]{Marquis_book}). 

\begin{corintro}
\phantomsection\label{corollary_HM_KM}
Let $G$ be a complete Kac--Moody group over a finite field   of order~$q$. Suppose that the Weyl group  is of compact hyperbolic type. If $q \geq 11$ (or $q \geq 4$ if $W$ is $3$-spherical), then  $G$ has the Howe--Moore property.
\end{corintro}

To the best of our knowledge, those are the first   examples of non-linear locally compact groups enjoying both Kazhdan's property (T) and the Howe--Moore property. Notice that a building $X$ of compact hyperbolic type is Gromov hyperbolic, hence the group $G$ of Corollary~\ref{corollary_HM_KM} is also Gromov hyperbolic, since it acts geometrically on its building.

If $(W, S)$ is of affine type and rank~$\geq 3$, then the conditions on the thickness of $X$ in Corollary~\ref{corollaryBB} can be softened: indeed, as noted above, Property~(T) is automatic as soon as $X$ is thick in the affine case. Hence we also obtain a new proof of the Howe--Moore property in the classical case, which is conceptually different from the previously known proofs  (see \cite{HoweMoore1979}, \cite[Theorem~2.2.20]{Zimmer1984} and \cite{Ciobotaru}).

\medskip
One of the main interests of proving the Howe--Moore property beyond the affine case is that it allows one to apply results of Creutz--Peterson \cite{CrePet, CrePet2} to  some of the simple Kac--Moody lattices from \cite{CapRem}, and deduce absence of non-trivial invariant random  subgroups (abbreviated IRS). Combining Corollary~\ref{corollary_HM_KM} with the work of Creutz--Peterson, we obtain the following statement, affording the first examples of finitely presented Kazhdan groups without IRS.

\begin{corintro}
\phantomsection\label{corollary_KM}
Let $\Lambda$ be a minimal simply connected split Kac--Moody group over a finite field of order $q$, and $Z$ be its center. Suppose that the Weyl group $W$ of $\Lambda$ is of compact hyperbolic type. If $q \geq 11$ (or $q \geq 4$ if $W$ is $3$-spherical), then   $\Lambda/Z$ has exactly two ergodic IRS's, namely $\delta_e$ and $\delta_{\Lambda/Z}$. 
\end{corintro}
	
 We refer to Corollary~\ref{cor:KM-sharp} for sharper bounds on $q$ depending on the type of the Weyl group $W$. 
 
 \medskip
The next application concerns unitary representations that are weakly contained in the left regular representation. In that case, the second item of Theorem~\ref{theoremA} is precluded because a parabolic subgroup is amenable if and only if it is of spherical type (see Proposition~\ref{prop:amen-parab} below).

\begin{corintro}
\phantomsection\label{corollaryC}
Let $X$, $G$ and $(W,S)$ be as in Theorem~\ref{theoremA}. Every unitary representation of $G$ weakly contained in the regular representation is mixing. 
\end{corintro}

As mentioned above, the existence of a non-compact proper open subgroup $O < G$ is an obstruction to the Howe--Moore property, since the quasi-regular representation $\lambda_{G/O}$ is not mixing. Applying our results to that representation yields the following. 

\begin{corintro}
\phantomsection\label{corollary_open}
Let $X$, $G$ and $(W,S)$ be as in Corollary~\ref{corollaryB}.  Then every non-compact open subgroup $O \leq G$ contains a conjugate of a finite index open subgroup of some non-spherical parabolic subgroup $P_J$. 
\end{corintro}

In the   case where $G$ is a complete Kac--Moody group over a finite field, a detailed description of open subgroups of $G$ is available,  without any restriction on $(W, S)$, see \cite{CapMar}.  In particular Corollary~\ref{corollary_open} is already known in this case. The proof of Corollary~\ref{corollary_open} afforded by the representation theoretic approach in this paper differs significantly from the proof in \cite{CapMar}, that relies on specific algebraic features of Kac--Moody groups.
	
\medskip
The paper is organized as follows. The ingredients on Iwahori--Hecke algebras and on amenability are collected in Section~\ref{sec:Hecke}. Those are used in Section~\ref{sec:chamber-decay} to give a criterion ensuring that the diagonal matrix coefficient defined by a $B$-fixed vector is $C_0$. The passage from $B$-fixed vectors to $B$-finite vectors is described in Section~\ref{sec:B-finite}, valid for arbitrary totally disconnected locally compact groups. Section~\ref{sec:decay} ends with the proof of Theorem~\ref{theoremA}. In Section~\ref{sec:Kazhdan}, we establish mild sufficient conditions on the thickness of a building of compact hyperbolic type to ensure that a Weyl-transitive automorphism group has Kazhdan's property (T).  This is then used to prove Corollaries~\ref{corollaryB}, \ref{corollaryBB} and \ref{corollary_KM}. Another  section is  devoted to the proofs of Corollaries~\ref{corollaryC} and~\ref{corollary_open} before a final section records some open questions.

\section{Hecke operators and amenability}\label{sec:Hecke}

We collect the ingredients used in the proof: some elementary computations in the Iwahori Hecke algebra and a spectral criterion for Bekka amenability.

\subsection{Weyl-transitive actions and Hecke operators}
\label{sec:hecke-preliminaries}

Let $\delta$ denote the Weyl distance on the set of chambers of $X$. Since $G$ is Weyl-transitive, for each $w\in W$ the set $\{g\in G:\delta(C_0,gC_0)=w\}$ is a single $B$-double coset, namely $BwB$. Local finiteness implies that every double coset $BwB$ is compact and has finite Haar measure. We normalize Haar measure $m$ on $G$ by imposing that $m(B)=1$. 

Under the assumption of Weyl-transitivity, the convolution algebra $C_c(B\backslash G/B)$ of complex-valued, continuous, compactly supported bi-$B$-invariant functions on $G$ is naturally isomorphic to the Iwahori--Hecke algebra of $W$ defined by the $|S|$-tuple $(q_s)_{s \in S}$, where $q_s=m(BsB)$, see \cite[\S19.1]{Davis} and \cite[CH.~IV, Exer.~23]{Bour}. Observe that $q_s$ is the number of left $B$-cosets in $BsB$, i.e. the number of chambers at Weyl distance $s$ from any given chamber $C$. In particular, every $s$-panel of $X$ contains exactly $q_s+1$  chambers. Since $X$ is thick by hypothesis, we have $q_s\geq2$ for all $s \in S$. 

The algebra $C_c(B\backslash G/B)$ admits  the set $(\mathbf 1_{BwB})_{ w \in W}$ of indicator functions of double $B$-cosets as a basis. We shall renormalize the basis vectors as follows. For $w\in W$, let $q_w=m(BwB)$ and let $T_w$ denote the  bi-$B$-invariant probability measure 
$$T_w=\frac 1 {q_w}\mathbf 1_{BwB}\,dm$$ 
supported on $BwB$. We shall identify the measure $T_w$ with its density; we may thus view $T_w$ as an element of $C_c(B\backslash G/B)$. 
Given  a unitary representation $\pi \colon G\to\cU(\cH)$, we have 
$$\pi(T_w)=q_w^{-1}\int_{BwB}\pi(g)\,dg,$$ and
$$\pi(T_w)^*=\pi(T_{w^{-1}}).$$ 
By \cite[Lemma~19.1.2]{Davis}, we have  $T_u*T_v=T_{uv}$ for all $u, v \in W$ with  $\ell(uv)=\ell(u)+\ell(v)$. In particular, if $w=s_1\cdots s_n$ is reduced, then $T_w=T_{s_1}*\cdots*T_{s_n}$.

In view of \cite[Proposition~19.1.1 and Lemma~19.1.5]{Davis}, the normalized quadratic relation in the Hecke algebra is
\begin{equation}\label{eq:quadratic}
        T_s*T_s=q_s^{-1}T_e+(1-q_s^{-1})T_s.
\end{equation}
Note that all structure constants in the normalized basis $(T_w)_{w\in W}$ are non-negative.

\begin{lemma}\label{lem:positive-group-product}
Let $t_1,\ldots,t_m\in S$. In the expansion of $T_{t_1}*\cdots*T_{t_m}$ in the normalized basis $(T_w)_{w\in W}$, the coefficient of $T_{t_1\cdots t_m}$ is strictly positive.
\end{lemma}

\begin{proof}
The proof is by induction on $m$. For $m=1$ there is nothing to prove. Suppose that the coefficient of $T_u$, where $u=t_1\cdots t_{m-1}$, is positive in $T_{t_1}*\cdots*T_{t_{m-1}}$. If $\ell(ut_m)=\ell(u)+1$, then $T_u*T_{t_m}=T_{ut_m}$. If $\ell(ut_m)=\ell(u)-1$, then the quadratic relation gives
\[
        T_u*T_{t_m}=q_{t_m}^{-1}T_{ut_m}+(1-q_{t_m}^{-1})T_u.
\]
Indeed, this can be established directly by conjugating the simple case, or by invoking  again \cite[Proposition~19.1.1]{Davis}. Thus in either case the coefficient of $T_{ut_m}=T_{t_1\cdots t_m}$ is positive. Since all other structure constants are non-negative, the claim follows.
\end{proof}

\subsection{The computation of the Markov operator}
\label{sec:markov-computation}

The following computation is the central point of the proof.

\begin{lemma}\label{lem:hecke-markov-parabolic}
Let $w=s_1\cdots s_n$ be a reduced word in $W$, and put $J=\supp(w)=\{s_1,\ldots,s_n\}\subset S$. Then $T_w^**T_w$ is a symmetric compactly supported probability measure on $P_J$. Moreover, the closed subgroup generated by $\supp(T_w^**T_w)$ is $P_J$.
\end{lemma}

\begin{proof}
Since $w\in W_J$, both $T_w$ and $T_w^*$ are supported on $P_J$, and hence $T_w^**T_w$ is a probability measure supported on $P_J$. It is symmetric because it has the form $\check\mu*\mu$.
It remains to show that its support generates $P_J$. 
Since $w=s_1\cdots s_n$ is reduced,
\[
        T_w^**T_w
        =T_{s_n}*\cdots*T_{s_1}*T_{s_1}*\cdots*T_{s_n}.
\]
For $1\leq i\leq n$, set
\[
        u_i=s_{i+1}\cdots s_n,
        \qquad
        r_i=u_i^{-1}s_i u_i=s_n\cdots s_{i+1}s_i s_{i+1}\cdots s_n.
\]
We claim that the basis element $T_{r_i}$ occurs with positive coefficient in the expansion of $T_w^**T_w$.
Indeed, expand the middle pairs successively. For $j<i$, choose the positive $T_e$-term in the quadratic relation \eqref{eq:quadratic} for $T_{s_j}*T_{s_j}$; for $j=i$, choose the positive $T_{s_i}$-term. What remains contains the product
\[
        T_{s_n}*\cdots*T_{s_{i+1}}*T_{s_i}*T_{s_{i+1}}*\cdots*T_{s_n}.
\]
By Lemma \ref{lem:positive-group-product}, the coefficient of $T_{r_i}$ in this product is positive. Since all structure constants are non-negative, $T_{r_i}$ occurs with positive coefficient in $T_w^**T_w$.

The elements $r_1,\ldots,r_n$ generate $W_J$. This follows by descending induction on $i$: $r_n=s_n$, and if $s_{i+1},\ldots,s_n$ are already in $\langle r_i,\ldots,r_n\rangle$, then $u_i$ is in that subgroup and $s_i=u_ir_i u_i^{-1}$ is also in it.

Since the double cosets $Br_iB$ occur in the support and the Weyl elements $r_i$ generate $W_J$, Lemma \ref{lem:positive-group-product} implies that the subgroup generated by the support contains $BuB$ for every $u\in W_J$. Hence it contains $\bigsqcup_{u\in W_J}BuB=P_J$. The reverse inclusion was already observed. Thus the closed subgroup generated by the support is $P_J$.
\end{proof}

\subsection{Amenability and a spectral criterion}
\label{sec:bekka}

We recall the two representation-theoretic facts used below. 

First, a unitary representation $\rho:G\to\cU(\cH)$ is amenable in Bekka's sense if and only if $\one_G\weakcontain \rho\otimes\ol\rho$; see Bekka \cite{BekkaAmenable} (see also  \cite[Appendix G]{BHV} for a detailed discussion of amenability). In particular, if $\rho$ contains a non-zero finite-dimensional subrepresentation, then $\one_G\leq  \rho\otimes\ol\rho$ by \cite[Proposition~A.1.12]{BHV}, so that $\rho$  is amenable in the sense of Bekka. 

Second, we shall use the following spectral criterion, see \cite[Proposition G.4.2]{BHV}.

\begin{proposition}[Bekka--de la Harpe--Valette]\label{prop:BHV}
Let $G$ be a locally compact group, let $\mu$ be an absolutely continuous probability measure on $G$, and let $\rho:G\to\cU(\cK)$ be a   unitary representation. Assume that the closed subgroup generated by $\supp(\check\mu*\mu)$ is $G$. If $\norm{\rho(\mu)}=1$, then $\one_G\weakcontain \rho$.
\end{proposition}

\begin{corollary}\label{cor:norm-one-bekka}
Let $G$ and $W$ be as in Theorem~\ref{theoremA}. Let $w\in W$, let $J=\supp(w)$, and let $\rho:P_J\to\cU(\cK)$ be a   unitary representation. If $\norm{\rho(T_w)}=1$, then $\one_{P_J}\weakcontain \rho$.
\end{corollary}

\begin{proof}
The measure $T_w$, viewed as a probability measure on $P_J$, is absolutely continuous with respect to Haar measure on $P_J$. By Lemma \ref{lem:hecke-markov-parabolic}, the support of $T_w^**T_w$ generates $P_J$. Proposition \ref{prop:BHV} applies.
\end{proof}

\section{Decay of matrix coefficients}\label{sec:decay}

This section contains the proof of chamber-spherical decay and its extension to $B$-finite vectors.

\subsection{Chamber-spherical decay for tensor representations}
\label{sec:chamber-decay}

Assume that $W$ is non-spherical.  Let $$M=\max\{\ell(u):u\in W_I,\ I\subset S,\ W_I\text{ spherical}\}.$$ The maximum is finite because $S$ is finite. Put $L=M+1$ and $\mathcal B_L=\{b\in W:\ell(b)=L\}$. Every $b\in\mathcal B_L$ has non-spherical support: otherwise $b$ would lie in a finite parabolic subgroup $W_{\supp(b)}$ and would have length at most $M$.

\begin{proposition}\label{prop:tensor-hecke-decay}
Let $\pi \colon G\to\cU(\cH)$ be a   unitary representation, and put $\sigma=\pi\otimes\ol\pi$. Assume that $\pi|_{P_J}$ is not amenable in the sense of Bekka for every minimal non-spherical subset $J\subset S$. Then $\norm{\sigma(T_w)}\longrightarrow0$ as $\ell(w)\to\infty$. In particular, all matrix coefficients of $\sigma$ defined by $B$-fixed vectors are $C_0$.  
\end{proposition}

\begin{proof}
Let $b\in\mathcal B_L$, and put $K=\supp(b)$. Then $K$ is non-spherical, so it contains a minimal non-spherical subset $J$. If $\norm{\sigma(T_b)}=1$, then Corollary \ref{cor:norm-one-bekka}, applied to $\sigma|_{P_K}$, gives $\one_{P_K}\weakcontain \sigma|_{P_K}$. Restricting to the closed subgroup $P_J<P_K$ yields $\one_{P_J}\weakcontain \sigma|_{P_J} = (\pi|_{P_J})\otimes\ol{(\pi|_{P_J})}$, contrary to the hypothesis. Hence $\norm{\sigma(T_b)}<1$ for every $b\in\mathcal B_L$.

Since $\mathcal B_L$ is finite, there is a number $\theta<1$ such that $\norm{\sigma(T_b)}\leq\theta$ for all $b\in\mathcal B_L$.
Choose a reduced expression for $w\in W$ and decompose it into consecutive blocks of length $L$ and a remainder:
\[
        w=b_1\cdots b_k r,
        \qquad b_i\in\mathcal B_L,
        \qquad \ell(r)<L,
\]
where all products are length-additive. Therefore $T_w=T_{b_1}*\cdots*T_{b_k}*T_r$. Hence $\norm{\sigma(T_w)}\leq \theta^k$. As $\ell(w)\to\infty$, one has $k \geq \ell(w)/L -1 \to\infty$, and the first assertion follows.

Let now $g_n\to\infty$ in $G$, and write $g_n\in Bw_nB$. Since $B$ is compact, we have $\ell(w_n)\to\infty$. 
If $\Xi,\Upsilon \in \cH \otimes \overline{\cH}$ are $B$-fixed vectors for $\sigma$, then
\[
        |\ip{\sigma(g_n)\Xi}{\Upsilon}|
        =|\ip{\sigma(T_{w_n})\Xi}{\Upsilon}| \leq \| \sigma(T_{w_n})\| \| \Xi \| \| \Upsilon \| 
        \to 0 
\]
in view of the first assertion. This confirms that all matrix coefficients of $\sigma$ defined by $B$-fixed vectors are $C_0$.  
\end{proof}

\subsection{Extension to $B$-finite vectors}
\label{sec:B-finite}

Our next goal is to show that if the diagonal matrix coefficient defined by a $B$-fixed vector in the tensor representation $\pi \otimes \ol \pi$ is $C_0$, then the representation $\pi$ is mixing. This part applies to an arbitrary pair $(G, B)$ consisting of a totally disconnected locally compact group $G$ and a compact open subgroup $B$, see Proposition~\ref{prop:Bfixed-tensor-implies-C0} below.

\medskip

Let $\cH$ be a Hilbert space. We denote by $\HS(\cH)$ the Hilbert space of Hilbert--Schmidt operators on $\cH$, with inner product $\ip{A}{B}_{\HS}=\sum_i \ip{A(e_i)}{B(e_i)}$, where $(e_i)$ is an orthonormal basis of $\cH$. The definition is independent of the chosen basis.

\begin{lemma}\label{lem:HS-projections}
Let $V_1,V_2$ be finite-dimensional subspaces of $\cH$, and let $E_i$ be the orthogonal projection onto $V_i$. Then, for all unit vectors $v_i\in V_i$, $\abs{\ip{v_1}{v_2}}^2\leq \ip{E_1}{E_2}_{\HS}$.
\end{lemma}

\begin{proof}
We have $\ip{E_1}{E_2}_{\HS}=\Tr(E_1E_2)\geq \dim(V_1\cap V_2)$. If $V_1\cap V_2\neq\{0\}$, the right-hand side is at least $1$, and the assertion is immediate.

Assume now that $V_1\cap V_2=\{0\}$. Choose a unit vector $w_2\in V_2$ maximizing $\norm{E_1v}$ among unit vectors $v\in V_2$.   
Then $\sup_{\norm{v_i}=1,\ v_i\in V_i}\abs{\ip{v_1}{v_2}}^2=\norm{E_1w_2}^2$. Choose an orthonormal basis $(e_i)$ of $\cH$ whose first vector is $w_2$ and whose first $\dim(V_2)$ vectors form an orthonormal basis of $V_2$. Then
\begin{align*}
        \abs{\ip{v_1}{v_2}}^2
        &\leq \norm{E_1w_2}^2
         = \norm{E_1e_1}^2                                      \\
        &\leq \sum_i \norm{E_1E_2e_i}^2
         = \Tr(E_2E_1E_2)
         = \Tr(E_1E_2)
         = \ip{E_1}{E_2}_{\HS}.
\end{align*}
This proves the claim.
\end{proof}

In the special case where $V_1$ and $V_2$ are both one-dimensional, the inequality of Lemma~\ref{lem:HS-projections} is an equality, see  \cite[Lemma~10.1]{EJK}. 

\begin{lemma}\label{lem:B-finite-bound}
Let $G$ be a   locally compact group,  $\pi \colon G\to\cU(\cH)$ be a unitary representation. Let $V_1,V_2$ be finite-dimensional subspaces of $\cH$, and let $E_i$ be the orthogonal projection onto $V_i$. Then, for all unit vectors $v_i\in V_i$ and all $g\in G$,
\[
        \abs{\ip{\pi(g)v_1}{v_2}}^2
        \leq
        \ip{\pi(g)E_1\pi(g)^{-1}}{E_2}_{\HS}.
\]
\end{lemma}

\begin{proof}
Apply Lemma \ref{lem:HS-projections} to the finite-dimensional subspaces $\pi(g)V_1$ and $V_2$. Their orthogonal projections are $\pi(g)E_1\pi(g)^{-1}$ and $E_2$.
\end{proof}

The right-hand side in Lemma \ref{lem:B-finite-bound} is a matrix coefficient of the conjugation representation $\Ad\pi \colon G\to\cU(\HS(\cH))$, defined by $(\Ad\pi)(g)A=\pi(g)A\pi(g)^{-1}$. Under the standard Hilbert-space identification $\HS(\cH)\cong\cH\otimes\ol\cH$, this is the tensor representation $\pi\otimes\ol\pi$. If $V_i$ is  invariant under a subgroup $B \leq G$, then $E_i$ is $B$-fixed for $\Ad\pi$.

Given a compact subgroup $B$ of a locally compact group $G$ and a unitary representation $\pi \colon G\to\cU(\cH)$, we say that a vector $v \in \cH$ is \textbf{$B$-finite} if the linear span of its $B$-orbit is finite-dimensional. It is a consequence of Peter--Weyl's theorem that the set of $B$-finite vectors is dense. Since the finite-dimensional representations of profinite groups factor through finite quotients, it follows that if $G$ is totally disconnected, the set of vectors with an open stabilizer is dense. This is standard and well-known,  see \cite[Lemma~1.1]{SemalOlshanskii}.

\begin{lemma}
\label{lem:finite-depth-HM}
Let $G$ be a   totally disconnected locally compact group and $(B_i)_{i \in I}$ be a basis of identity neighborhoods consisting of compact open subgroups. 
Let 
\(\pi \colon G\to\mathcal U(\mathcal H)\) be a unitary representation such that for all $i \in I$  and every net
\(g_n\to\infty\) in \(G\), one has
\(P_i^\pi\pi(g_n)P_i^\pi\to 0\) weakly, where $P^{\pi}_i$ denotes the orthogonal projection onto the $B_i$-fixed vectors.
Then $\pi$ is mixing.
\end{lemma}

\begin{proof}
Let \(\pi \colon G\to\mathcal U(\mathcal H)\) be a unitary
representation. Let \(\xi,\eta\in\mathcal H\), and let \(\varepsilon>0\). As recalled above, the set of  vectors with an open stabilizer is  dense in
\(\mathcal H\). We can thus choose vectors
\(\xi_0,\eta_0\in\mathcal H\) with an open stabilizer that are sufficiently close to $\xi, \eta$ so that
$
\|\xi-\xi_0\|\,\|\eta\|+\|\xi_0\|\,\|\eta-\eta_0\|
<
\varepsilon/2.
$
Then, for every \(g\in G\), we have
\[
\begin{aligned}
\left|
\langle \pi(g)\xi,\eta\rangle
-
\langle \pi(g)\xi_0,\eta_0\rangle
\right|
&\leq
\left|\langle \pi(g)(\xi-\xi_0),\eta\rangle\right|
+
\left|\langle \pi(g)\xi_0,\eta-\eta_0\rangle\right| \\
&\leq
\|\xi-\xi_0\|\,\|\eta\|
+
\|\xi_0\|\,\|\eta-\eta_0\| < \varepsilon/2.
\end{aligned}
\]

By hypothesis, there is $i \in I$ such that $B_i$ fixes both $\xi_0$ and $\eta_0$. 
Thus \(\xi_0,\eta_0\in\mathcal H^{B_i}\). If \(g_n\to\infty\) in \(G\),
then the hypothesis gives
$
P_i^\pi\pi(g_n)P_i^\pi\to 0
$
weakly. Therefore
$$
\langle \pi(g_n)\xi_0,\eta_0\rangle
=
\left\langle
P_i^\pi\pi(g_n)P_i^\pi \xi_0,\eta_0
\right\rangle
\to 0.
$$
Consequently there is a compact subset \(C\subset G\) such that
$
\left|\langle \pi(g)\xi_0,\eta_0\rangle\right|
<
\varepsilon/2
$
for all \(g\notin C\). For all such \(g\), the uniform estimate above gives
$
\left|\langle \pi(g)\xi,\eta\rangle\right|
<
\varepsilon.
$
This proves that every matrix coefficient of \(\pi\) vanishes at infinity. 
\end{proof}

\begin{proposition}\label{prop:Bfixed-tensor-implies-C0}
Let $G$ be a totally disconnected locally compact group, let $B<G$ be compact open, and let $\pi \colon G\to\cU(\cH)$ be a   unitary representation. Put $\sigma=\pi\otimes\ol\pi$. Assume that for every net $g_n\to\infty$ in $G$ and all $B$-fixed vectors $\Xi,\Upsilon\in\cH\otimes\ol\cH$, one has $\ip{\sigma(g_n)\Xi}{\Upsilon}\to0$. Then $\pi$ is mixing.
\end{proposition}

\begin{proof}
Let $g_n\to\infty$ in $G$. We first show that matrix coefficients of $\pi$ defined by $B$-finite vectors vanish along $(g_n)$. 

We choose a basis  $(B_i)_{i \in I}$   of identity neighborhoods consisting of compact open subgroups with $B_i \leq B$ for all $i$. Fix $i \in I$ and let $v_1, v_2 \in \cH$ be unit vectors that are both fixed by $B_i$. For $j=1, 2$, put $V_j=\operatorname{span}(Bv_j)$. We have  $\dim(V_j) \leq [B:B_i]$, hence $V_j$ is finite-dimensional and $B$-invariant. 

Let $E_j$ denote the orthogonal projection onto $V_j$. The projector $E_j$ is $B$-fixed for the conjugation representation $\Ad\pi$, identified with $\sigma$. Lemma~\ref{lem:B-finite-bound} gives
\[
        \abs{\ip{\pi(g_n)v_1}{v_2}}^2
        \leq
        \ip{\pi(g_n)E_1\pi(g_n)^{-1}}{E_2}_{\HS}
        \longrightarrow0.
\]
Thus all matrix coefficients of $\pi$ defined by $B_i$-fixed vectors tend to zero along every net escaping compact subsets of $G$. Therefore, the hypotheses of Lemma~\ref{lem:finite-depth-HM} are fulfilled. It follows that $\pi$ is mixing.
\end{proof}

\subsection{Proof of Theorem~\ref{theoremA}}

We are now ready to finish the proof of  the main result.

\begin{proof}[Proof of Theorem~\ref{theoremA}]
If $W$ is spherical, then $X$ has finite diameter. Since $X$ is locally finite and has finite rank, the chamber set is finite. Hence $G$ is compact, and every continuous matrix coefficient belongs to $C_0(G)$. Thus the first alternative holds.

We may assume that $W$ is non-spherical. Suppose that the second alternative in Theorem~\ref{theoremA} does not hold. Thus, for every minimal non-spherical $J\subset S$, one has $\one_{P_J}\not\weakcontain (\pi|_{P_J})\otimes\ol{(\pi|_{P_J})}$. Put $\sigma=\pi\otimes\ol\pi$. By Proposition~\ref{prop:tensor-hecke-decay}, all   matrix coefficients of $\sigma$ defined by $B$-fixed vectors are $C_0$. Hence the hypotheses of Proposition~\ref{prop:Bfixed-tensor-implies-C0} are satisfied. It follows that  $\pi$ is mixing, as required. 
\end{proof}

\section{Property (T) in the compact hyperbolic case}
\label{sec:Kazhdan}

It is known by the work of Dymara--Januszkiewicz~\cite{DymJan} that property (T) holds for groups acting Weyl-transitively on locally finite buildings of $2$-spherical type, as soon as the thickness is large enough. Our next goal is to provide lower bounds on the thickness in the case where the type is compact hyperbolic, that are as good as possible. We rely on the well-known classification of Coxeter diagrams of compact hyperbolic type, and on Kassabov's criterion~\cite{Ka11}.

Let $X$ be a thick locally finite building of finite rank and type $(W,S)$. We retain the notation of Section~\ref{sec:hecke-preliminaries}. Thus for each $s \in S$, every $s$-panel of $X$ contains exactly $q_s +1$ chambers. We set 
$$\qmin = \min\{q_s \mid s \in S\}.$$
Assume that $(W, S)$ is of compact hyperbolic type and rank $r = |S|$. If in addition $(W, S)$ is crystallographic (i.e. all Coxeter numbers are in $\{2, 3, 4, 6, \infty\}$), then $r \in\{3, 4, 5\}$. Moreover, there are only two possible Coxeter diagrams of rank~$4$, and a single one of rank~$5$, see \cite[Ch.~V, \S4, Ex.~15]{Bour}. In all cases, the diagram is circular or linear. We may therefore number the elements of $S$ by the integers modulo $r$, and assume that $m_{i, j} = 2$ for all $i, j$ with $|i-j| \geq 2$. Setting $m_i = m_{i-1, i}$ for all $i \mod r$, we see that the Coxeter diagram of $(W, S)$ is completely determined by the $r$-tuple $(m_1, \dots, m_r)$, that we shall call the \textbf{type} of $(W, S)$.

\begin{proposition}\label{prop:Kazhdan}
Let $X$ be a thick locally finite building of finite rank whose Weyl group is of compact hyperbolic type   $\mathcal D = (m_1, \dots, m_r)$, and let $G<\Aut(X)$ be a closed type-preserving subgroup acting Weyl-transitively. Suppose that one of the following conditions holds:

\begin{enumerate}[label=(\arabic*)]
\item $\mathcal D \in \{(2, 4, 6), (3,3,4), (3, 3, 3, 4), (3, 3, 3, 3, 4)\}$ and $\qmin \geq 3$. 
\item $\mathcal D \in \{(2, 6, 6), (3, 4, 3,4)\}$ and $\qmin \geq 4$. 

\item $\mathcal D \in \{(3, 4, 4), (3,3,6)\}$ and $\qmin \geq 5$. 

\item $\mathcal D \in \{(3, 4, 6), (4,4,4)\}$ and $\qmin \geq 6$. 

\item $\mathcal D = (3,6,6)$ and $\qmin \geq 7$. 
\item $\mathcal D = (4, 4, 6)$   and $\qmin \geq 8$. 

\item $\mathcal D = (4,6,6)$ and $\qmin \geq 9$. 

\item $\mathcal D = (6,6,6)$ and $\qmin \geq 10$. 
\end{enumerate}
	
Then $G$ has Kazhdan's property (T). 
\end{proposition}

Before presenting the proof, we need to recall Kassabov's criterion~\cite{Ka11}. 
Given a group $H$ generated by two subgroups $A, B \leq H$ 
 and a unitary representation $\pi$ of $H$ without non-zero $H$-invariant vectors, we define
$$\varepsilon_H(A, B; \pi) = \sup\big\{\frac{|\langle u, v \rangle |}{\| u \| \| v \|} \mid u \in \mathcal H_\pi^A \setminus \{0\}, v \in \mathcal H_\pi^B \setminus \{0\}\big\}.
$$
 If $\mathcal H_\pi^A = \{0\}$ or $\mathcal H_\pi^B = \{0\}$, we put $\varepsilon_H(A, B; \pi) =0$. If $\pi$ is an arbitrary unitary representation, we define  $\varepsilon_H(A, B; \pi)$ as $\varepsilon_H(A, B;\pi_0)$, where $\pi_0 \leq \pi$ is the sub-representation of $\pi$ defined on the orthogonal complement of the subspace of $H$-invariant vectors. 
The supremum of $\varepsilon_H(A, B; \pi) $ taken over all unitary representations $\pi$ of $H$ with $\mathcal H_\pi^H = \{0\}$, is denoted by 
$\varepsilon_H(A, B).$

Let us now consider a group $G$ generated by compact subgroups $A_1, \dots, A_n$, so that $A_{i, j} = \langle A_i \cup A_j \rangle$ is compact for all $i, j$. Set $\varepsilon_{i, j} = \varepsilon_{A_{i, j}}(A_i, A_j)$. Let $M$ be the symmetric $n\times n$-matrix defined by $M_{i, i} = 1$ and $M_{i, j}= -\varepsilon_{i, j}$  for all $i \neq j$. Kassabov's criterion \cite[\S2]{Ka11} ensures that if $M$ is positive definite, then $G$ has property (T).

\begin{proof}[Proof of Proposition~\ref{prop:Kazhdan}]
The group $G$ is generated by the rank~one parabolic subgroups $P_s = B \sqcup BsB$, that are compact. Moreover, the rank~two parabolics $P_{s, t} = \langle P_s \cup P_t\rangle$ are also compact since the Weyl group $W$ is $2$-spherical. Let $m_{s, t}$ be the order of $st$. We set
$$\varepsilon_{s, t} = \left\{
\begin{array}{ll}
0 & \text{if } m_{s, t} = 2,\\
\displaystyle \left( \frac{q_s + q_t + 2\sqrt{q_s q_t}\cos(2\pi/m_{s, t})}{(q_s+1)(q_t+1)}\right)^{\frac 1 2}
&\text{if } m_{s, t} \geq 3.
\end{array}
\right.
$$
By \cite[Proposition~7.1]{Garland} and \cite[Theorem~2.10(ii)]{CKKW}, we have $\varepsilon_{P_{s, t}}(P_s, P_t) = \varepsilon_{s, t}$. 

By hypothesis, the Weyl group $W$ is of type $(m_1, \dots, m_r)$. Let $S = \{s_1, s_2, \dots, s_{r}\}$ be the ordering of $S$ so that $m_{s_{i-1}, s_i} = m_i$, and set $\varepsilon_i = \varepsilon_{s_{i-1}, s_i}$. In view of Kassabov's criterion, in order to prove that $G$ has (T), it suffices to ensure that the $r\times r$-matrix 
$$M =  \left(
\begin{array}{cccccc}
1 & -\varepsilon_1 & 0 & \dots & 0 & -\varepsilon_r \\ 
-\varepsilon_1 & 1 & -\varepsilon_2 & \dots & 0 & 0 \\
0 & -\varepsilon_2 & 1 & \dots & 0 & 0 \\
\vdots & \vdots & \vdots &  \ddots & \vdots & \vdots  \\ 
0 & 0 & 0 & \dots &  1 & -\varepsilon_{r-1} \\
-\varepsilon_r & 0 & 0 & \dots & -\varepsilon_{r-1} & 1 
\end{array}
\right).
$$
is positive definite. 

Observe that $\varepsilon_{s, t}$, viewed as a function of $q_s \geq 2$ (resp. $q_t \geq 2$),  is  strictly decreasing. In particular the quantity $\varepsilon_{s, t}$ does not decrease if one replaces $q_s$ and $q_t$ by $\qmin$. 

Assume first that  $r=3$. In that case, the matrix $M$   is positive definite if and only if $\varepsilon_1^2+\varepsilon_2^2+\varepsilon_3^2 + 2\varepsilon_1\varepsilon_2\varepsilon_3 < 1$, see \cite[\S4.2]{Ka11} (see also \cite[Theorem~5.9]{EJ}). A direct verification shows that the latter inequality indeed holds under the hypotheses of the proposition. 

Assume now that $r \geq 4$. In particular  we have $m_i \geq 3$ for all $i$, so that $\varepsilon_i >0$ for all $i$. This allows us to invoke  \cite[Proposition~2.5]{CapKas}.  More precisely, \cite[Proposition~2.5(i)]{CapKas} ensures that $M$ is positive definite provided $\max_i \{\varepsilon_{i-1}+ \varepsilon_i\} < 1$. A direct verification shows that the latter inequality holds if $\qmin \geq 4$. In particular the proposition is proved if (2) holds. It remains to check that $M$ is positive definite if $\mathcal D \in \{   (3, 3, 3, 4), (3, 3, 3, 3, 4)\}$ and $\qmin = 3$. This follows from \cite[Proposition 2.5(iii)]{CapKas}.
\end{proof}
	
\begin{remark} \label{rem:lowthickness}

One can compute numerically that, in the case where $q_s = \qmin$ for all $s$, the sufficient conditions for the matrix $M$ in the proof of Proposition~\ref{prop:Kazhdan} to be positive definite are also necessary. It would be interesting to determine whether the group $G$ has Kazhdan's property (T) in those cases. There is one case where it is already known to fail: indeed, let \(X\) be the building associated with an
RGD-system of type \((4,4,4)\) over \(\mathbf F_2\).  Every panel of \(X\)
then contains three chambers, but \(\Aut(X)\) does not have
property~\textup{(T)}; see \cite[Lemma~7.4.6]{BischofThesis}.  The proof uses
the positive unipotent subgroup \(U_+\), which is a non-finitely generated
lattice in \(\Aut(X)\). See also Section~\ref{sec:outlook} for further discussion.
\end{remark}

\begin{corollary}\label{cor:minimal-non-sph:Kazhdan}
Let $X$ be a thick locally finite building of finite rank and of minimal non-spherical type $(W, S)$. Let $G<\Aut(X)$ be a closed type-preserving subgroup acting Weyl-transitively. Suppose that one of the following conditions holds. 
\begin{enumerate}[label=(\arabic*)]
\item $(W, S)$ is simply laced. 
\item $(W, S)$ is $3$-spherical, and every panel of $X$  contains at least $5$ chambers. 
\item $(W, S)$ is $2$-spherical crystallographic, and every panel of $X$ contains  at least $11$ chambers.
\end{enumerate}

Then $G$ has Kazhdan's property (T). 
\end{corollary}
\begin{proof}
By hypothesis $(W, S)$ is of minimal non-spherical type, hence it is irreducible. 
There are two cases. If  $(W, S)$ is of affine type, then property (T) holds by \cite{Oppenheim}. Otherwise $(W, S)$ is of compact hyperbolic type. In that case, it cannot be simply laced.  Moreover, in the $3$-spherical case, every irreducible residue of rank~$3$ is endowed with a strongly transitive action (see \cite[Proposition~6.15]{AbramenkoBrown}). Hence, it follows from Tits' classification \cite[Theorem~11.4]{Tits74} that $(W, S)$ is crystallographic. Therefore, the type of $(W, S)$ is one of those appearing in  Proposition~\ref{prop:Kazhdan}. We infer that $G$ has property (T) if $\qmin \geq 10$ in the rank~$3$ case, and if $\qmin \geq 4$ if the rank is~$4$ or~$5$. 
\end{proof}

We are now ready to complete the proofs of Corollaries~\ref{corollaryB}--\ref{corollary_KM}.

\begin{proof}[Proof of Corollary~\ref{corollaryB}]
Suppose that $\pi$ is not mixing. By Theorem~\ref{theoremA}, there is a minimal non-spherical subset $J\subset S$ such that $\pi|_{P_J}$ is amenable, i.e. $\one_{P_J}\weakcontain(\pi|_{P_J})\otimes\ol{(\pi|_{P_J})}$. Since $P_J$ has property (T) by Corollary~\ref{cor:minimal-non-sph:Kazhdan}, this weak containment is actual containment. Thus $(\pi|_{P_J})\otimes\ol{(\pi|_{P_J})}$ has a non-zero invariant vector and hence $\pi|_{P_J}$ has a finite-dimensional subrepresentation, see \cite[Proposition~A.1.12]{BHV}.
\end{proof}

The following result appears in  \cite[Corollary~3.1]{CapMon} for strongly transitive actions. Strong transitivity is only used to invoke Tits' transitivity lemma, which is actually valid under the hypothesis of Weyl transitivity, see \cite[\S6.2.7]{AbramenkoBrown}. Hence, the arguments from \cite[Corollary~3.1]{CapMon} yield the following. 

\begin{proposition}\label{prop:simple-by-compact}
Let $X$ be a thick locally finite building of irreducible non-spherical type, and $G \leq \Aut(X)$ be a closed subgroup acting Weyl-transitively. Then the intersection $G^+$ of all non-trivial closed normal subgroups of $G$ is topologically simple, and the quotient $G/G^+$ is compact. 
\end{proposition}

We may now deduce Corollary~\ref{corollaryBB} from Corollary~\ref{corollaryB}. 

\begin{proof}[Proof of Corollary~\ref{corollaryBB}]
Since $(W, S)$ is of minimal non-spherical type, it is irreducible. 
Let  $\pi $ be a unitary representation of $G$. Suppose first that $\pi$ contains a finite-dimensional subrepresentation $\rho$. Since $G$ is totally disconnected and $\rho$ takes its values in a finite-dimensional Lie group, the kernel of $\rho$ is open. Thus $\rho$ is trivial on the subgroup $G^+$ from Proposition~\ref{prop:simple-by-compact}. It follows that $\rho$ factors through a compact quotient of $G$. By hypothesis, this implies that $\rho$ is the trivial representation. Hence $\pi$ has non-zero  invariant vectors. Now the Howe--Moore property for $G$ follows directly from the dichotomy asserted by Corollary~\ref{corollaryB}. 
\end{proof}

\begin{remark}
  \label{rem:Kunze-Stein}
In the situation of Corollary~\ref{corollaryBB}, the estimates above also give a weak form of the Kunze--Stein phenomenon for $B$-bi-invariant functions. We record this only as a byproduct. 

Let \(m_B\) denote the normalized Haar probability measure on \(B\). It coincides with the restriction to $B$ of the Haar measure $m$ on $G$. The operator
$
\lambda(m_B*\delta_{\dot w}*m_B)
$
on \(L^2(G)\) is the regular-representation instance of the normalized
chamber Hecke operator \(T_w\). Hence the proof of Proposition~\ref{prop:tensor-hecke-decay}
gives constants \(N\geq 1\) and \(0<\theta<1\) such that
\[
\bigl\|\lambda(m_B*\delta_{\dot w}*m_B)\bigr\|_{B(L^2(G))}
\leq
\theta^{\lfloor \ell(w)/N\rfloor}
\]
for all \(w\in W\).
It follows that convolution by \(B\)-bi-invariant functions satisfies a
Kunze--Stein estimate for \(p\) sufficiently close to \(1\). More precisely,
there exists \(p_0>1\) such that, for every \(1\leq p<p_0\), there is a
constant \(C_p\) with
$
\|f*\xi\|_2\leq C_p\|f\|_p\|\xi\|_2
$
for every compactly supported bi-\(B\)-invariant function \(f\) on \(G\) and
every \(\xi\in L^2(G)\).

Indeed, write
$
f=\sum_{w\in W} a_w\,1_{B\dot wB}
$
with finite support. Since
$
1_{B\dot wB}=q_w T_w = q_w \,m_B*\delta_{\dot w}*m_B,
$
we get
\[
\|\lambda(f)\|
\leq
\sum_{w\in W}|a_w|\, q_w\,
\theta^{\lfloor \ell(w)/N\rfloor}.
\]
On the other hand,
$
\|f\|_p^p=\sum_{w\in W}|a_w|^p q_w.
$
Thus H\"older's inequality gives
\[
\|\lambda(f)\|
\leq
\|f\|_p
\left(
\sum_{w\in W}
\theta^{p'\lfloor \ell(w)/N\rfloor} q_w
\right)^{1/p'},
\]
where \(p'\) is conjugate to \(p\). The last series is finite for \(p\)
sufficiently close to \(1\), because \(W\) has at most exponential growth
and \(q_w\) grows at most exponentially in \(\ell(w)\).
\end{remark}

\begin{proof}[Proof of Corollary~\ref{corollary_HM_KM}]
It follows from \cite[Prop.~11 and Cor.~16]{CapRem} that the group $G$ is topologically simple (it is even abstractly simple, see \cite{Mar}). In particular, its only compact quotient is trivial. Moreover $G$ acts strongly transitively (hence Weyl-transitively) on its building. Hence the conclusion follows from Corollary~\ref{corollaryBB}. 
\end{proof}

\section{Proof of Corollary~\ref{corollary_KM}}

An \textbf{invariant random subgroup} (or \textbf{IRS}) of a locally compact group \(G\)
is a conjugation-invariant Borel probability measure on the space of closed
subgroups of \(G\) (equipped with the Chabauty topology). Equivalently, it
is the distribution of stabilizers for a probability-measure-preserving
action of \(G\). We denote by \(\IRS(G)\) the space of all invariant random
subgroups of \(G\), and by \(\IRSerg(G)\) the subspace of ergodic ones.

Recall also that a character of a countable group \(\Gamma\) is a
conjugation-invariant positive definite function normalized by
\(\chi(e)=1\); it is extremal if it is an extreme point of the convex set
of characters. If \(\theta\in\IRS(\Gamma)\), then the associated
fixed-point character is given by
$
\chi_\theta(g)=\theta\bigl(\{H\leq \Gamma:g\in H\}\bigr).
$
If \(\theta\) is realized as the stabilizer distribution of a probability-measure-preserving 
action \(\Gamma\curvearrowright (Y,\mu)\), then
\(\chi_\theta(g)=\mu(\{y\in Y:gy=y\})\).

\medskip

We shall need the following auxiliary fact. 

\begin{lemma}\label{lem:convergence-radius}
Let $(W, S)$ be a Coxeter system of crystallographic compact hyperbolic type. Then the convergence radius of the growth series of $W$ is strictly larger than $1/2$. 
\end{lemma}
\begin{proof}
As mentioned above, the rank of $(W, S)$ is $3, 4$ or $5$. In the rank three case, the result follows from~\cite[Theorem~A]{Bischof}.

In general, it is known that the growth series of $W$ is a rational fraction,  so that its convergence radius is the smallest root of the denominator,  see \cite[\S17.1]{Davis}. For Coxeter groups of compact hyperbolic type and rank~$4, 5$, the growth series are explicitly known, see the tables in \cite{CLS}.  For the three crystallographic Coxeter diagrams, one checks numerically that the smallest root of the denominator is~$>1/2$. 
\end{proof}

We are now ready to apply the Howe--Moore property through the stabilizer
rigidity theorem of Creutz and Peterson. We use the following special case of \cite[Theorem~9.1]{CrePet}.

\begin{theorem}[Creutz--Peterson {\cite[Theorem~9.1]{CrePet}}]
\label{thm:CP-character}
Let \(\Gamma<L_1\times L_2\) be an irreducible lattice, where \(L_1,L_2\)
are second countable, non-discrete, non-compact, topologically simple,
totally disconnected, locally compact, Kazhdan groups with the Howe--Moore property.
Then any ergodic probability-measure-preserving action of $\Gamma$ either has finite orbits or has
finite stabilizers.
\end{theorem}

Corollary~\ref{corollary_KM} from the introduction is a direct consequence of the following more precise statement.

\begin{corollary}\label{cor:KM-sharp}
Let $\Lambda$ be a minimal simply connected split Kac--Moody group over a finite field of order $q$, and $Z$ be its center. Suppose that the Weyl group $W$ of $\Lambda$ is of compact hyperbolic type $\mathcal D = (m_1, \dots, m_r)$. Suppose that $q = \qmin$ satisfies one of the conditions (1)--(7) of Proposition~\ref{prop:Kazhdan}. Then every invariant random subgroup of  $\Lambda/Z$ is a convex combination of the trivial ones, namely $\delta_e$ and $\delta_{\Lambda/Z}$
\end{corollary}
\begin{proof}
Let \(\Delta=\Delta_+\times\Delta_-\) be the twin building associated with
\(\Lambda\), and let \(L_\pm\) be the closures of the images of
\(\Lambda\) in \(\operatorname{Aut}(\Delta_\pm)\). Let also $Z$ be the center of $\Lambda$, which is finite.  

 First of all, we need to make sure that   the central quotient $\Lambda/Z$ embeds as a lattice in the automorphism group of its twin buildings. 
The minimal Kac--Moody group $\Lambda/Z$ is an irreducible lattice in the automorphism group of its twin buildings if and only if $\mathbf W(1/q)$ converges, where $\mathbf W$ is the  growth series of the Coxeter group $W$, see \cite{RemyCRAS}. Under our hypotheses on the type of $W$, this is indeed the case for all $q \geq 2$ by Lemma~\ref{lem:convergence-radius}.
By
\cite[Theorem~15 and Corollary~16]{CapRem}, the group \(\Lambda/Z\) is   infinite
and simple. Moreover the groups \(L_+\) and \(L_-\) are second countable,
non-discrete, non-compact, topologically simple,  totally disconnected
 locally compact groups acting strongly transitively on
\(\Delta_+\) and \(\Delta_-\), respectively.  Each building \(\Delta_\pm\) has compact-hyperbolic
crystallographic type and all panels have thickness \(q+1\).  It follows from Proposition~\ref{prop:Kazhdan} that $L_+$ and $L_-$ have Kazhdan's property~(T), and from  Corollary~\ref{corollary_HM_KM} that they have the Howe--Moore property.

Theorem~\ref{thm:CP-character} therefore applies to
\(\Lambda/Z<L_+\times L_-\). 
Let \(\theta\in\IRSerg(\Lambda/Z)\), and realize it as the stabilizer
distribution of an ergodic probability-measure-preserving action
\(\Lambda/Z \curvearrowright (Y,\mu)\). 
Hence,  this action is either has finite stabilizers or finite
orbits. 
In the first case, there are only countably many finite subgroups, so that $\theta$ must be concentrated on a finite conjugacy class of a finite subgroup $H < \Lambda/Z$. Hence $N_{\Lambda/Z}(H)$ has finite index in the simple group $\Gamma$. Thus, $N_{\Lambda/Z}(H) = \Gamma$ and hence $H$ is normal in $\Gamma.$ Since $H$ is finite, we conclude that $H$ is trivial and $\theta = \delta_{e}$.
In the finite-orbit case, almost every stabilizer has finite index in
\(\Lambda/Z\). Since \(\Lambda\) is infinite and simple, it has no non-trivial
finite quotients, so every finite orbit is a singleton. Hence
\(\theta=\delta_{\Lambda/Z}\). The result follows. 
\end{proof}

\begin{remark}
  It seems plausible that a lattice $\Lambda/Z$ also has no non-trivial characters. One possible way to establish this would be through  \cite[Corollary~6.6]{CrePet2}. This requires checking that the lattice $\Lambda_+ \leq  L_+$ defined as the projection of the intersection $\Lambda/Z \cap (L_+ \times B_-)$, is square-integrable, where $B_-$ is a compact open subgroup of $L_-$.  It is known that $\Lambda/Z$ itself is a square-integrable lattice in $L_+\times L_-$, see \cite{RemyIntegrability}, but this does not imply that $\Lambda_+$ is square-integrable: indeed $\Lambda_+$ can fail to be finitely generated, see Remark~\ref{rem:lowthickness}.  The condition of square-integrability of a lattice is related to the non-distortion of the lattice in its ambient locally compact group, see \cite[Theorem~1.3]{CapRem_nondist}. It is known that $\Lambda/Z$ is non-distorted in $L_+ \times L_-$, whereas in the hyperbolic case, the lattice $\Lambda_+$, which is non-uniform, is always distorted in $L_+$, see \cite[Theorem~1.1 and Section~3.3]{CapRem_nondist}. 
\end{remark}

\section{Proof of Corollaries~\ref{corollaryC} and \ref{corollary_open}}
\label{sec:regular}

\begin{proposition}\label{prop:amen-parab}
Let $X$ be a thick locally finite building of finite rank and type $(W,S)$, and let $G<\Aut(X)$ be a closed type-preserving subgroup acting Weyl-transitively. For every parabolic subgroup $P_J \leq G$, the following conditions are equivalent. 
\begin{enumerate}[label=(\roman*)]
\item $J$ is spherical. 
\item $P_J$ is compact. 
\item $P_J$ is amenable. 
\item $P_J$ does not contain any non-abelian discrete free subgroup. 
\end{enumerate}
\end{proposition}
\begin{proof}
The only non-obvious implication is from (iv) to (i). Let us assume that $J$ is non-spherical. We shall prove that $P_J$ contains a non-abelian discrete free subgroup. Upon replacing $J$ by a subset, we may and shall assume that $J$ is irreducible and minimal non-spherical. 

Let $R_J \subset X$ be a $J$-residue whose stabilizer is $P_J$. Thus $P_J$ acts Weyl-transitively on $R_J$. Observe that $|R_J|$ is a closed convex subset in the Davis realization $|X|$ of the building $X$, which is CAT($0$), see \cite[\S18.3]{Davis}. Since $P_J$ is generated by the parabolic subgroups of rank~one that it contains, all of which are compact, it follows that $P_J$ is unimodular. 

Since $J$  is irreducible and minimal non-spherical, there are two cases: either it is affine, or it is compact hyperbolic. If $J$ is affine of rank two, or compact hyperbolic, then   the building $|R_J|$ is CAT($-1$). Since $P_J$ is unimodular, it follows from \cite[Theorem~M]{CapMon_amen} that it does not fix any point at infinity of $|R_J|$. Therefore, it contains two loxodromic isometries with distinct fixed points at infinity. The Ping-Pong Lemma shows that any sufficiently large positive powers of those two isometries generate a non-abelian free subgroup, that acts freely on $|R_J|$ and is thus a discrete subgroup of $P_J$. 

Assume now that $J$ is affine of rank at least three. By \cite[Theorem~1.2]{CapCio}, the group $P_J$ contains a regular element. The fact that $P_J$ acts Weyl-transitively on $R_J$ implies that it acts strongly transitively, see \cite[Theorem~5.16]{Kramer}. Hence the action on the spherical building at infinity is strongly transitive. We  infer that  $P_J$ contains two regular elements whose attracting and repelling fixed chambers at infinity are mutually opposite. Two such elements form a Ping-Pong pair, hence they have powers that generate a discrete non-abelian free subgroup, see the proof of Theorem~8.10 in  \cite{BalBri}. 

In the affine case, one may also conclude alternatively by relying on the classical Tits alternative, as follows. Since the $P_J$-action on $R_J$ is strongly transitive, it follows from \cite[Corollary~E]{CapMon_Bieberbach} that  the building $R_J$ is Bruhat--Tits.  Therefore, the full automorphism group $\Aut(R_J)$ has a cocompact normal subgroup $L_J$ that is an isotropic simple linear algebraic group over a local field, see   \cite[Proposition~9.1]{BCL}.  Such a linear group $L_J$ contains non-abelian discrete free subgroups. On the other hand, the natural map $P_J \to \Aut(R_J)$ has a compact kernel $K$. Moreover the quotient $P_J/K$ is simple-by-compact in view of Proposition~\ref{prop:simple-by-compact}. It follows that the image of $P_J$ in $ \Aut(R_J)$ contains $L_J$. Hence $P_J$ contains a non-abelian discrete free subgroup.
\end{proof}

\begin{proof}[Proof of Corollary~\ref{corollaryC}]
Let $\pi\weakcontain\lambda_G$, and suppose that $\pi$ is not  mixing. By Theorem~\ref{theoremA}, there is a minimal non-spherical subset $J\subset S$ such that, with $P=P_J$, one has $\one_P\weakcontain (\pi|_P)\otimes\ol{(\pi|_P)} \weakcontain (\lambda_G|_P) \otimes \ol{(\lambda_G|_P)} \sim \lambda_P$ by   \cite[Corollary~E.2.6(ii)]{BHV}. Hence $\one_P\weakcontain\lambda_P$ and  $P$ is amenable by Hulanicki's criterion \cite{Hulanicki}. Thus $P$ is of spherical type by Proposition~\ref{prop:amen-parab}, a contradiction. Therefore $\pi$ is mixing.
\end{proof}

The following elementary fact is well-known. 

\begin{lemma}\label{lem:fd}
Let $\Omega$ be a set and $H$ be a group acting on $\Omega$ and let $\pi$ be the permutation representation of $H$ on $\ell^2(\Omega)$. Then $\pi$ contains a   finite-dimensional subrepresentation if and only if $\Omega$ has a finite orbit. 
\end{lemma}
\begin{proof}
By \cite[Proposition~A.1.12]{BHV}, $\pi$ contains a finite-dimensional subrepresentation if and only if $\pi\otimes\ol\pi$ has a non-zero invariant vector. This happens if and only if the action of $H$ on $\Omega \times \Omega$ has a finite orbit if and only if the action of $H$ on $\Omega$ has a finite orbit.
\end{proof}

\begin{proof}[Proof of Corollary~\ref{corollary_open}]
Let $O \leq G$ be a non-compact open subgroup.  
 Then the quasi-regular representation    $\pi$   of $G$ on $\ell^2(G/O)$ is not mixing, since the indicator function $\one_O(g)$ is a diagonal matrix coefficient of $\pi$ that is not $C_0$. By Corollary~\ref{corollaryB}, this implies that there is a minimal non-spherical subset $J \subset S$ such that $\pi|_{P_J}$ contains a finite-dimensional subrepresentation of $P_J$. In view of Lemma~\ref{lem:fd}, we infer that some $P_J$-orbit on the coset space $G/O$ is finite. Hence, for some $g \in G$ and some   finite index open  subgroup of $P'_J \leq P_J$, we have $gP_J' g^{-1} \leq O$. 
\end{proof}

\section{Further directions and open problems} \label{sec:outlook}

As mentioned in the introduction, the two assertions in  Theorem~\ref{theoremA} are not mutually exclusive. We have seen in Corollary~\ref{corollaryB} that, in $2$-spherical type, they do become exclusive provided the thickness is moderately large, see Section~\ref{sec:Kazhdan}. It is natural to wonder what happens when the thickness is small.
 
Using work of Ashcroft~\cite{Ashcroft}, we can show that Bischof's example discussed   in Remark~\ref{rem:lowthickness} does in fact not only fail to have property (T) but it also has the Haagerup property:

\begin{proposition}\label{prop:haagerup-444}
Let
$
(G,(U_\alpha)_{\alpha\in\Phi})
$
be an RGD-system of type $(4,4,4)$ over $\mathbf F_2$, and let
$
\Delta=(\Delta_+,\Delta_-,\delta_*)
$
be its associated twin building. Then the locally compact groups
$\Aut(\Delta_\pm)$ have the Haagerup property.
\end{proposition}
\begin{proof}
Put $X=\Delta_-$. Bischof observes
that the buildings $X$ obtained in this way from all RGD-systems of
type $(4,4,4)$ over $\mathbf F_2$ are mutually isomorphic; see
\cite[Remark~1(b)]{BischofUncountably}. In particular, $X$ is
isomorphic to the building associated with a split Kac--Moody group
of type $(4,4,4)$ over $\mathbf F_2$. That building admits a cocompact lattice $\Gamma$ acting simply transitively on its vertices;
see \cite[p.~102]{CKV}. Thus $\Gamma$ is a uniform lattice in
$\Aut(X)$.

By \cite[Corollary~B]{Ashcroft}, the group $\Gamma$ is virtually special
and, in particular, has the Haagerup property. Since the Haagerup property
is invariant under passage between a locally compact group and a lattice
\cite[Proposition~6.1.5]{CCJJV}, it follows that $\Aut(X)$ has the Haagerup
property.
\end{proof}

Thus, apart from the  case of trees mentioned in the introduction, the buildings in the previous proposition provide  examples of buildings whose automorphism group does not satisfy a strict dichotomy in Theorem~\ref{theoremA}. It would be interesting to know if there are other examples of this kind -- and if in fact automorphism groups of buildings of minimal non-spherical type do always either have property (T) \emph{or}  the Haagerup property.

Another natural question arising from this work is whether every \textit{irreducible} infinite-dimensional unitary representation of a group $G$ as in Theorem~\ref{theoremA} is non-amenable in the sense of Bekka. This would ensure that the conditions on the thickness in Corollary~\ref{corollaryBB} can be replaced by the sole requirement that $X$ is thick: indeed, every irreducible non-trivial representation of $G$ would then be mixing by Theorem~\ref{theoremA}, hence $G$ would have the Howe--Moore property by \cite[Proposition~2.3.2]{Zimmer1984}. 

\section*{Acknowledgments}

This paper is the result of joint efforts after a first version was circulated by the second author. This first version contained a proof of the Howe--Moore property for groups of minimal non-spherical type and large thickness, including the applications to  rigidity of invariant random subgroups. When writing the first version, OpenAI's GPT was used for substantial interactive research and writing assistance.        

The second author thanks Uri Bader for helpful comments, especially asking about Remark~\ref{rem:Kunze-Stein}.
	
\bibliographystyle{abbrv} 
\bibliography{Mixing}

@article {EJK,
    AUTHOR = {Ershov, Mikhail and Jaikin-Zapirain, Andrei and Kassabov,
              Martin},
     TITLE = {Property {$(T)$} for groups graded by root systems},
   JOURNAL = {Mem. Amer. Math. Soc.},
  FJOURNAL = {Memoirs of the American Mathematical Society},
    VOLUME = {249},
      YEAR = {2017},
    NUMBER = {1186},
     PAGES = {v+135},
      ISSN = {0065-9266,1947-6221},
      ISBN = {978-1-4704-2604-0; 978-1-4704-4139-5},
   MRCLASS = {22D10 (17B22 17B70 20E42)},
  MRNUMBER = {3724373},
MRREVIEWER = {Dave\ Witte\ Morris},
       DOI = {10.1090/memo/1186},
       URL = {https://doi.org/10.1090/memo/1186},
}

@unpublished{BischofUncountably,
  author        = {Bischof, Sebastian},
  title         = {Uncountably many $2$-spherical groups of
                   {Kac--Moody} type of rank $3$ over {$\mathbb{F}_2$}},
  year          = {2025},
  note        = {Preprint, arXiv:2504.17513},
}

@article{CKV,
  author  = {Carbone, Lisa and Kangaslampi, Riikka and Vdovina, Alina},
  title   = {Groups acting simply transitively on vertex sets of
             hyperbolic triangular buildings},
  journal = {LMS J. Comput. Math.},
  volume  = {15},
  year    = {2012},
  pages   = {101--112},
  doi     = {10.1112/S1461157012000083},
  url     = {https://doi.org/10.1112/S1461157012000083}
}

@article{Ashcroft,
  author  = {Ashcroft, Calum J.},
  title   = {Link conditions for cubulation},
  journal = {Int. Math. Res. Not. IMRN},
  year    = {2023},
  number  = {12},
  pages   = {9950--10012},
  doi     = {10.1093/imrn/rnac111},
  url     = {https://doi.org/10.1093/imrn/rnac111}
}

@book{CCJJV,
  author    = {Cherix, Pierre-Alain and Cowling, Michael and
               Jolissaint, Paul and Julg, Pierre and Valette, Alain},
  title     = {Groups with the {Haagerup} Property:
               {Gromov}'s a-{T}-menability},
  series    = {Progress in Mathematics},
  volume    = {197},
  publisher = {Birkh{\"a}user Verlag},
  address   = {Basel},
  year      = {2001}
}

@article {Ka11,
    AUTHOR = {Kassabov, Martin},
     TITLE = {Subspace arrangements and property {T}},
   JOURNAL = {Groups Geom. Dyn.},
  FJOURNAL = {Groups, Geometry, and Dynamics},
    VOLUME = {5},
      YEAR = {2011},
    NUMBER = {2},
     PAGES = {445--477},
      ISSN = {1661-7207,1661-7215},
   MRCLASS = {20F65 (20F55 22D10)},
  MRNUMBER = {2782180},
MRREVIEWER = {Pierre-Emmanuel\ Caprace},
       DOI = {10.4171/GGD/134},
       URL = {https://doi.org/10.4171/GGD/134},
}

@article {HoweMoore1979,
    AUTHOR = {Howe, Roger E. and Moore, Calvin C.},
     TITLE = {Asymptotic properties of unitary representations},
   JOURNAL = {J. Functional Analysis},
  FJOURNAL = {Journal of Functional Analysis},
    VOLUME = {32},
      YEAR = {1979},
    NUMBER = {1},
     PAGES = {72--96},
      ISSN = {0022-1236},
   MRCLASS = {22E50 (22D10 43A30 43A80)},
  MRNUMBER = {533220},
MRREVIEWER = {Jonathan\ M.\ Rosenberg},
       DOI = {10.1016/0022-1236(79)90078-8},
       URL = {https://doi.org/10.1016/0022-1236(79)90078-8},
}

@incollection {LubotzkyMozes1992,
    AUTHOR = {Lubotzky, Alexander and Mozes, Shahar},
     TITLE = {Asymptotic properties of unitary representations of tree
              automorphisms},
 BOOKTITLE = {Harmonic analysis and discrete potential theory ({F}rascati,
              1991)},
     PAGES = {289--298},
 PUBLISHER = {Plenum, New York},
      YEAR = {1992},
      ISBN = {0-306-44225-6},
   MRCLASS = {22D40 (20E08 58F11)},
  MRNUMBER = {1222467},
MRREVIEWER = {Alexander\ Starkov},
}

@article {BuMo_lattices,
    AUTHOR = {Burger, Marc and Mozes, Shahar},
     TITLE = {Lattices in product of trees},
   JOURNAL = {Inst. Hautes \'Etudes Sci. Publ. Math.},
  FJOURNAL = {Institut des Hautes \'Etudes Scientifiques. Publications
              Math\'ematiques},
    NUMBER = {92},
      YEAR = {2000},
     PAGES = {151--194},
      ISSN = {0073-8301,1618-1913},
   MRCLASS = {20E08 (22E40)},
  MRNUMBER = {1839489},
MRREVIEWER = {Alexander\ Lubotzky},
       URL = {http://www.numdam.org/item?id=PMIHES_2000__92__151_0},
}

@article {Ciobotaru,
    AUTHOR = {Ciobotaru, Corina},
     TITLE = {A unified proof of the {H}owe-{M}oore property},
   JOURNAL = {J. Lie Theory},
  FJOURNAL = {Journal of Lie Theory},
    VOLUME = {25},
      YEAR = {2015},
    NUMBER = {1},
     PAGES = {65--89},
      ISSN = {0949-5932},
   MRCLASS = {22D10 (20E42)},
  MRNUMBER = {3345827},
MRREVIEWER = {Alain\ Valette},
}

@book {AbramenkoBrown,
    AUTHOR = {Abramenko, Peter and Brown, Kenneth S.},
     TITLE = {Buildings},
    SERIES = {Graduate Texts in Mathematics},
    VOLUME = {248},
      NOTE = {Theory and applications},
 PUBLISHER = {Springer, New York},
      YEAR = {2008},
     PAGES = {xxii+747},
      ISBN = {978-0-387-78834-0},
   MRCLASS = {20E42 (20F55 20J06 51E24 51F15)},
  MRNUMBER = {2439729},
MRREVIEWER = {Ralf\ Koehl},
       DOI = {10.1007/978-0-387-78835-7},
       URL = {https://doi.org/10.1007/978-0-387-78835-7},
}

@article {BekkaAmenable,
    AUTHOR = {Bekka, B.},
     TITLE = {Amenable unitary representations of locally compact groups},
   JOURNAL = {Invent. Math.},
  FJOURNAL = {Inventiones Mathematicae},
    VOLUME = {100},
      YEAR = {1990},
    NUMBER = {2},
     PAGES = {383--401},
      ISSN = {0020-9910,1432-1297},
   MRCLASS = {22D10 (43A07 46L99)},
  MRNUMBER = {1047140},
MRREVIEWER = {T.\ Sund},
       DOI = {10.1007/BF01231192},
       URL = {https://doi.org/10.1007/BF01231192},
}

@book {BHV,
    AUTHOR = {Bekka, Bachir and de la Harpe, Pierre and Valette, Alain},
     TITLE = {Kazhdan's property ({T})},
    SERIES = {New Mathematical Monographs},
    VOLUME = {11},
 PUBLISHER = {Cambridge University Press, Cambridge},
      YEAR = {2008},
     PAGES = {xiv+472},
      ISBN = {978-0-521-88720-5},
   MRCLASS = {22-02 (22E40 28D15 37A15 43A07 43A35)},
  MRNUMBER = {2415834},
MRREVIEWER = {Markus\ Neuhauser},
       DOI = {10.1017/CBO9780511542749},
       URL = {https://doi.org/10.1017/CBO9780511542749},
}

@article {SemalOlshanskii,
    AUTHOR = {Semal, Lancelot},
     TITLE = {Unitary representations of totally disconnected locally
              compact groups satisfying {O}l'\v shanski\u i's
              factorisation},
   JOURNAL = {Represent. Theory},
  FJOURNAL = {Representation Theory. An Electronic Journal of the American
              Mathematical Society},
    VOLUME = {27},
      YEAR = {2023},
     PAGES = {356--414},
      ISSN = {1088-4165},
   MRCLASS = {22D10 (20E08 51E24 57M07)},
  MRNUMBER = {4609145},
MRREVIEWER = {Robert\ S.\ Doran},
       DOI = {10.1090/ert/637},
       URL = {https://doi.org/10.1090/ert/637},
}

@article {Hulanicki,
    AUTHOR = {Hulanicki, Andrzej},
     TITLE = {Means and {F}\o lner condition on locally compact groups},
   JOURNAL = {Studia Math.},
  FJOURNAL = {Polska Akademia Nauk. Instytut Matematyczny. Studia
              Mathematica},
    VOLUME = {27},
      YEAR = {1966},
     PAGES = {87--104},
      ISSN = {0039-3223,1730-6337},
   MRCLASS = {22.65 (42.50)},
  MRNUMBER = {195982},
MRREVIEWER = {M.\ M.\ Day},
       DOI = {10.4064/sm-27-2-87-104},
       URL = {https://doi.org/10.4064/sm-27-2-87-104},
}

@article {CapKas,
    AUTHOR = {Caprace, Pierre-Emmanuel and Kassabov, Martin},
     TITLE = {Tame automorphism groups of polynomial rings with property
              ({T}) and infinitely many alternating group quotients},
   JOURNAL = {Trans. Amer. Math. Soc.},
  FJOURNAL = {Transactions of the American Mathematical Society},
    VOLUME = {376},
      YEAR = {2023},
    NUMBER = {11},
     PAGES = {7983--8021},
      ISSN = {0002-9947,1088-6850},
   MRCLASS = {22D55 (05C48 14E07 20D06 20F67)},
  MRNUMBER = {4657226},
MRREVIEWER = {Piotr\ Mizerka},
       DOI = {10.1090/tran/8988},
       URL = {https://doi.org/10.1090/tran/8988},
}

@article {Oppenheim,
    AUTHOR = {Oppenheim, Izhar},
     TITLE = {Property ({T}) for groups acting on affine buildings},
   JOURNAL = {Bull. Lond. Math. Soc.},
  FJOURNAL = {Bulletin of the London Mathematical Society},
    VOLUME = {57},
      YEAR = {2025},
    NUMBER = {10},
     PAGES = {3151--3162},
      ISSN = {0024-6093,1469-2120},
   MRCLASS = {20F65 (22D55)},
  MRNUMBER = {4971894},
       DOI = {10.1112/blms.70148},
       URL = {https://doi.org/10.1112/blms.70148},
}

@article {CapMar,
    AUTHOR = {Caprace, Pierre-Emmanuel and Marquis, Timoth\'ee},
     TITLE = {Open subgroups of locally compact {K}ac-{M}oody groups},
   JOURNAL = {Math. Z.},
  FJOURNAL = {Mathematische Zeitschrift},
    VOLUME = {274},
      YEAR = {2013},
    NUMBER = {1-2},
     PAGES = {291--313},
      ISSN = {0025-5874,1432-1823},
   MRCLASS = {20E42 (20F55)},
  MRNUMBER = {3054330},
MRREVIEWER = {Jean\ L\'ecureux},
       DOI = {10.1007/s00209-012-1070-4},
       URL = {https://doi.org/10.1007/s00209-012-1070-4},
}

@article {CapRem,
    AUTHOR = {Caprace, Pierre-Emmanuel and R\'emy, Bertrand},
     TITLE = {Simplicity and superrigidity of twin building lattices},
   JOURNAL = {Invent. Math.},
  FJOURNAL = {Inventiones Mathematicae},
    VOLUME = {176},
      YEAR = {2009},
    NUMBER = {1},
     PAGES = {169--221},
      ISSN = {0020-9910,1432-1297},
   MRCLASS = {20G44 (20E42 20G15 51E24)},
  MRNUMBER = {2485882},
MRREVIEWER = {Guy\ Rousseau},
       DOI = {10.1007/s00222-008-0162-6},
       URL = {https://doi.org/10.1007/s00222-008-0162-6},
}

@article {CrePet,
    AUTHOR = {Creutz, Darren and Peterson, Jesse},
     TITLE = {Stabilizers of ergodic actions of lattices and commensurators},
   JOURNAL = {Trans. Amer. Math. Soc.},
  FJOURNAL = {Transactions of the American Mathematical Society},
    VOLUME = {369},
      YEAR = {2017},
    NUMBER = {6},
     PAGES = {4119--4166},
      ISSN = {0002-9947,1088-6850},
   MRCLASS = {22D40 (22E40 22F10 28D15 37A15)},
  MRNUMBER = {3624404},
MRREVIEWER = {Bruno\ Duchesne},
       DOI = {10.1090/tran/6836},
       URL = {https://doi.org/10.1090/tran/6836},
}

@book {Davis,
    AUTHOR = {Davis, Michael W.},
     TITLE = {The geometry and topology of {C}oxeter groups},
    SERIES = {London Mathematical Society Monographs Series},
    VOLUME = {32},
 PUBLISHER = {Princeton University Press, Princeton, NJ},
      YEAR = {2008},
     PAGES = {xvi+584},
      ISBN = {978-0-691-13138-2; 0-691-13138-4},
   MRCLASS = {20F55 (05B45 05C25 51-02 57M07)},
  MRNUMBER = {2360474},
MRREVIEWER = {Ralf\ Koehl},
}

@book {Bour,
    AUTHOR = {Bourbaki, N.},
     TITLE = {\'El\'ements de math\'ematique. {F}asc. {XXXIV}. {G}roupes et
              alg\`ebres de {L}ie. {C}hapitre {IV}: {G}roupes de {C}oxeter
              et syst\`emes de {T}its. {C}hapitre {V}: {G}roupes engendr\'es
              par des r\'eflexions. {C}hapitre {VI}: syst\`emes de racines},
    SERIES = {Actualit\'es Scientifiques et Industrielles [Current
              Scientific and Industrial Topics]},
    VOLUME = {No. 1337},
 PUBLISHER = {Hermann, Paris},
      YEAR = {1968},
     PAGES = {288 pp. (loose errata)},
   MRCLASS = {22.50 (17.00)},
  MRNUMBER = {240238},
MRREVIEWER = {G.\ B.\ Seligman},
}

@article {CrePet2,
    AUTHOR = {Creutz, Darren and Peterson, Jesse},
     TITLE = {Character rigidity for lattices and commensurators},
   JOURNAL = {Amer. J. Math.},
  FJOURNAL = {American Journal of Mathematics},
    VOLUME = {146},
      YEAR = {2024},
    NUMBER = {3},
     PAGES = {687--711},
      ISSN = {0002-9327,1080-6377},
   MRCLASS = {22E40 (20G25 22D55)},
  MRNUMBER = {4752678},
MRREVIEWER = {B.\ Sury},
       DOI = {10.1353/ajm.2024.a928322},
       URL = {https://doi.org/10.1353/ajm.2024.a928322},
}

@article {CapMon,
    AUTHOR = {Caprace, Pierre-Emmanuel and Monod, Nicolas},
     TITLE = {Decomposing locally compact groups into simple pieces},
   JOURNAL = {Math. Proc. Cambridge Philos. Soc.},
  FJOURNAL = {Mathematical Proceedings of the Cambridge Philosophical
              Society},
    VOLUME = {150},
      YEAR = {2011},
    NUMBER = {1},
     PAGES = {97--128},
      ISSN = {0305-0041,1469-8064},
   MRCLASS = {22D05 (20E32)},
  MRNUMBER = {2739075},
MRREVIEWER = {George\ A.\ Willis},
       DOI = {10.1017/S0305004110000368},
       URL = {https://doi.org/10.1017/S0305004110000368},
}

@book {Zimmer1984,
    AUTHOR = {Zimmer, Robert J.},
     TITLE = {Ergodic theory and semisimple groups},
    SERIES = {Monographs in Mathematics},
    VOLUME = {81},
 PUBLISHER = {Birkh\"auser Verlag, Basel},
      YEAR = {1984},
     PAGES = {x+209},
      ISBN = {3-7643-3184-4},
   MRCLASS = {22E40 (22D40 28D15)},
  MRNUMBER = {776417},
MRREVIEWER = {S.\ G.\ Dani},
       DOI = {10.1007/978-1-4684-9488-4},
       URL = {https://doi.org/10.1007/978-1-4684-9488-4},
}

@article {DymJan,
    AUTHOR = {Dymara, Jan and Januszkiewicz, Tadeusz},
     TITLE = {Cohomology of buildings and their automorphism groups},
   JOURNAL = {Invent. Math.},
  FJOURNAL = {Inventiones Mathematicae},
    VOLUME = {150},
      YEAR = {2002},
    NUMBER = {3},
     PAGES = {579--627},
      ISSN = {0020-9910,1432-1297},
   MRCLASS = {20E42 (20F55 20J06 22E50)},
  MRNUMBER = {1946553},
MRREVIEWER = {Alain\ Valette},
       DOI = {10.1007/s00222-002-0242-y},
       URL = {https://doi.org/10.1007/s00222-002-0242-y},
}

@article {Bischof,
    AUTHOR = {Bischof, Sebastian},
     TITLE = {On growth functions of {C}oxeter groups},
   JOURNAL = {Proc. Edinb. Math. Soc. (2)},
  FJOURNAL = {Proceedings of the Edinburgh Mathematical Society. Series II},
    VOLUME = {68},
      YEAR = {2025},
    NUMBER = {3},
     PAGES = {979--993},
      ISSN = {0013-0915,1464-3839},
   MRCLASS = {20F55 (20F69 51F15)},
  MRNUMBER = {4943977},
MRREVIEWER = {Anthony\ Genevois},
       DOI = {10.1017/S0013091525000094},
       URL = {https://doi.org/10.1017/S0013091525000094},
}

@phdthesis{BischofThesis,
  author = {Bischof, Sebastian},
  title  = {Construction of {RGD}-systems of type $(4,4,4)$ over
            ${\mathbb F_2}$},
  school = {Justus-Liebig-Universit{\"a}t Gie{\ss}en},
  year   = {2023},
  url    = {https://jlupub.ub.uni-giessen.de/items/4e38869c-b9ea-4c74-a9fb-beb2056d2e6f}
}

@book {Tits74,
    AUTHOR = {Tits, Jacques},
     TITLE = {Buildings of spherical type and finite {BN}-pairs},
    SERIES = {Lecture Notes in Mathematics},
    VOLUME = {Vol. 386},
 PUBLISHER = {Springer-Verlag, Berlin-New York},
      YEAR = {1974},
     PAGES = {x+299},
   MRCLASS = {20G15 (50A20)},
  MRNUMBER = {470099},
MRREVIEWER = {Bruce\ Cooperstein},
}

@article {CKKW,
    AUTHOR = {Caprace, Pierre-Emmanuel and Conder, Marston and Kaluba, Marek
              and Witzel, Stefan},
     TITLE = {Hyperbolic generalized triangle groups, property ({T}) and
              finite simple quotients},
   JOURNAL = {J. Lond. Math. Soc. (2)},
  FJOURNAL = {Journal of the London Mathematical Society. Second Series},
    VOLUME = {106},
      YEAR = {2022},
    NUMBER = {4},
     PAGES = {3577--3637},
      ISSN = {0024-6107,1469-7750},
   MRCLASS = {20F67 (20G44 22D55 57K20)},
  MRNUMBER = {4524205},
MRREVIEWER = {Olga\ Varghese},
       DOI = {10.1112/jlms.12668},
       URL = {https://doi.org/10.1112/jlms.12668},
}

@article {Garland,
    AUTHOR = {Garland, Howard},
     TITLE = {{$p$}-adic curvature and the cohomology of discrete subgroups
              of {$p$}-adic groups},
   JOURNAL = {Ann. of Math. (2)},
  FJOURNAL = {Annals of Mathematics. Second Series},
    VOLUME = {97},
      YEAR = {1973},
     PAGES = {375--423},
      ISSN = {0003-486X},
   MRCLASS = {20J05},
  MRNUMBER = {320180},
MRREVIEWER = {M.\ S.\ Raghunathan},
       DOI = {10.2307/1970829},
       URL = {https://doi.org/10.2307/1970829},
}

@article {Kramer,
    AUTHOR = {Kramer, Linus},
     TITLE = {Some remarks on proper actions, proper metric spaces, and
              buildings},
   JOURNAL = {Adv. Geom.},
  FJOURNAL = {Advances in Geometry},
    VOLUME = {22},
      YEAR = {2022},
    NUMBER = {4},
     PAGES = {541--559},
      ISSN = {1615-715X,1615-7168},
   MRCLASS = {54H15 (20E42)},
  MRNUMBER = {4497184},
MRREVIEWER = {Hendrik\ Van Maldeghem},
       DOI = {10.1515/advgeom-2022-0018},
       URL = {https://doi.org/10.1515/advgeom-2022-0018},
}

@article {CapMon_amen,
    AUTHOR = {Caprace, Pierre-Emmanuel and Monod, Nicolas},
     TITLE = {Fixed points and amenability in non-positive curvature},
   JOURNAL = {Math. Ann.},
  FJOURNAL = {Mathematische Annalen},
    VOLUME = {356},
      YEAR = {2013},
    NUMBER = {4},
     PAGES = {1303--1337},
      ISSN = {0025-5831,1432-1807},
   MRCLASS = {43A07 (20F65 47H10)},
  MRNUMBER = {3072802},
MRREVIEWER = {Mehdi\ Sangani Monfared},
       DOI = {10.1007/s00208-012-0879-9},
       URL = {https://doi.org/10.1007/s00208-012-0879-9},
}

@article {CapMon_Bieberbach,
    AUTHOR = {Caprace, Pierre-Emmanuel and Monod, Nicolas},
     TITLE = {An indiscrete {B}ieberbach theorem: from amenable {$\rm
              CAT(0)$} groups to {T}its buildings},
   JOURNAL = {J. \'Ec. polytech. Math.},
  FJOURNAL = {Journal de l'\'Ecole polytechnique. Math\'ematiques},
    VOLUME = {2},
      YEAR = {2015},
     PAGES = {333--383},
      ISSN = {2429-7100,2270-518X},
   MRCLASS = {53C23 (20E42 20F65 43A07 51E24 53C20)},
  MRNUMBER = {3438730},
MRREVIEWER = {Fernando\ Galaz-Garc\'ia},
       DOI = {10.5802/jep.26},
       URL = {https://doi.org/10.5802/jep.26},
}

@article {BCL,
    AUTHOR = {Bader, Uri and Caprace, Pierre-Emmanuel and L\'ecureux, Jean},
     TITLE = {On the linearity of lattices in affine buildings and
              ergodicity of the singular {C}artan flow},
   JOURNAL = {J. Amer. Math. Soc.},
  FJOURNAL = {Journal of the American Mathematical Society},
    VOLUME = {32},
      YEAR = {2019},
    NUMBER = {2},
     PAGES = {491--562},
      ISSN = {0894-0347,1088-6834},
   MRCLASS = {20E42 (20C99 20E08 20F65 22D40 22E40 22F50 51E24)},
  MRNUMBER = {3904159},
MRREVIEWER = {Herbert\ Abels},
       DOI = {10.1090/jams/914},
       URL = {https://doi.org/10.1090/jams/914},
}

@article {CapCio,
    AUTHOR = {Caprace, Pierre-Emmanuel and Ciobotaru, Corina},
     TITLE = {Gelfand pairs and strong transitivity for {E}uclidean
              buildings},
   JOURNAL = {Ergodic Theory Dynam. Systems},
  FJOURNAL = {Ergodic Theory and Dynamical Systems},
    VOLUME = {35},
      YEAR = {2015},
    NUMBER = {4},
     PAGES = {1056--1078},
      ISSN = {0143-3857,1469-4417},
   MRCLASS = {51E24 (22D15)},
  MRNUMBER = {3345164},
MRREVIEWER = {Linus\ K. H. Kramer},
       DOI = {10.1017/etds.2013.102},
       URL = {https://doi.org/10.1017/etds.2013.102},
}

@article {BalBri,
    AUTHOR = {Ballmann, Werner and Brin, Michael},
     TITLE = {Orbihedra of nonpositive curvature},
   JOURNAL = {Inst. Hautes \'Etudes Sci. Publ. Math.},
  FJOURNAL = {Institut des Hautes \'Etudes Scientifiques. Publications
              Math\'ematiques},
    NUMBER = {82},
      YEAR = {1995},
     PAGES = {169--209},
      ISSN = {0073-8301,1618-1913},
   MRCLASS = {53C23 (53C20 57M20)},
  MRNUMBER = {1383216},
MRREVIEWER = {Ralf\ J.\ Spatzier},
       URL = {http://www.numdam.org/item?id=PMIHES_1995__82__169_0},
}

@article {RemyCRAS,
    AUTHOR = {R{\'e}my, Bertrand},
     TITLE = {Construction de r\'eseaux en th\'eorie de {K}ac-{M}oody},
   JOURNAL = {C. R. Acad. Sci. Paris S\'er. I Math.},
  FJOURNAL = {Comptes Rendus de l'Acad\'emie des Sciences. S\'erie I.
              Math\'ematique},
    VOLUME = {329},
      YEAR = {1999},
    NUMBER = {6},
     PAGES = {475--478},
      ISSN = {0764-4442},
   MRCLASS = {20E42},
  MRNUMBER = {1715140},
MRREVIEWER = {Paolo\ Casati},
       DOI = {10.1016/S0764-4442(00)80044-0},
       URL = {https://doi.org/10.1016/S0764-4442(00)80044-0},
}

@article {RemyIntegrability,
    AUTHOR = {R{\'e}my, Bertrand},
     TITLE = {Integrability of induction cocycles for {K}ac--{M}oody groups},
   JOURNAL = {Math. Ann.},
    VOLUME = {333},
      YEAR = {2005},
    NUMBER = {1},
     PAGES = {29--43},
}

@article {CLS,
    AUTHOR = {Chapovalov, Maxim and Leites, Dimitry and Stekolshchik,
              Rafael},
     TITLE = {The {P}oincar\'e{} series of the hyperbolic {C}oxeter groups
              with finite volume of fundamental domains},
   JOURNAL = {J. Nonlinear Math. Phys.},
  FJOURNAL = {Journal of Nonlinear Mathematical Physics},
    VOLUME = {17},
      YEAR = {2010},
     PAGES = {169--215},
      ISSN = {1402-9251,1776-0852},
   MRCLASS = {20F55},
  MRNUMBER = {2827480},
MRREVIEWER = {Adam\ Piggott},
       DOI = {10.1142/S1402925110000842},
       URL = {https://doi.org/10.1142/S1402925110000842},
}

@article {EJ,
    AUTHOR = {Ershov, Mikhail and Jaikin-Zapirain, Andrei},
     TITLE = {Property ({T}) for noncommutative universal lattices},
   JOURNAL = {Invent. Math.},
  FJOURNAL = {Inventiones Mathematicae},
    VOLUME = {179},
      YEAR = {2010},
    NUMBER = {2},
     PAGES = {303--347},
      ISSN = {0020-9910,1432-1297},
   MRCLASS = {22D10 (20E08 20E18 20F69)},
  MRNUMBER = {2570119},
MRREVIEWER = {Markus\ Neuhauser},
       DOI = {10.1007/s00222-009-0218-2},
       URL = {https://doi.org/10.1007/s00222-009-0218-2},
}

@article {Mar,
    AUTHOR = {Marquis, Timoth\'ee},
     TITLE = {Abstract simplicity of locally compact {K}ac-{M}oody groups},
   JOURNAL = {Compos. Math.},
  FJOURNAL = {Compositio Mathematica},
    VOLUME = {150},
      YEAR = {2014},
    NUMBER = {4},
     PAGES = {713--728},
      ISSN = {0010-437X,1570-5846},
   MRCLASS = {20G44 (20E42)},
  MRNUMBER = {3200675},
MRREVIEWER = {Walter\ D.\ Freyn},
       DOI = {10.1112/S0010437X13007598},
       URL = {https://doi.org/10.1112/S0010437X13007598},
}

@book {Marquis_book,
    AUTHOR = {Marquis, Timoth\'ee},
     TITLE = {An introduction to {K}ac-{M}oody groups over fields},
    SERIES = {EMS Textbooks in Mathematics},
 PUBLISHER = {European Mathematical Society (EMS), Z\"urich},
      YEAR = {2018},
     PAGES = {xi+331},
      ISBN = {978-3-03719-187-3},
   MRCLASS = {22-02 (17B99 20E42 20F55 20G44)},
  MRNUMBER = {3838421},
MRREVIEWER = {Ralf\ Koehl},
       DOI = {10.4171/187},
       URL = {https://doi.org/10.4171/187},
}

@article {CapRem_nondist,
    AUTHOR = {Caprace, Pierre-Emmanuel and R\'emy, Bertrand},
     TITLE = {Non-distortion of twin building lattices},
   JOURNAL = {Geom. Dedicata},
  FJOURNAL = {Geometriae Dedicata},
    VOLUME = {147},
      YEAR = {2010},
     PAGES = {397--408},
      ISSN = {0046-5755,1572-9168},
   MRCLASS = {20E42 (20F55 20F65 20G44 51E24)},
  MRNUMBER = {2660586},
MRREVIEWER = {Ralf\ Koehl},
       DOI = {10.1007/s10711-010-9469-8},
       URL = {https://doi.org/10.1007/s10711-010-9469-8},
}

\end{document}